\documentclass[a4paper]{amsart}
\usepackage{srcltx,bm}

\usepackage{varioref}
\usepackage{hyperref}
\hypersetup{colorlinks=true,pageanchor=false,
linkcolor=blue,citecolor=red,urlcolor=red}
\usepackage[all]{hypcap}

\allowdisplaybreaks
\newcommand*{\refh}[2]{\hyperref[#2]{#1~\ref{#2}}}

\newcounter{ourcount}
\usepackage[T1]{fontenc}
\usepackage[utf8x]{inputenc}
\usepackage{yfonts}
\usepackage{dsfont}
\usepackage{amscd,amssymb,amsmath,amsthm,amsfonts}
\usepackage{mathtools}
\usepackage{accents}
\usepackage{graphicx}
\usepackage{mathrsfs}
\usepackage[all,cmtip]{xy}
\usepackage{tikz}
\usetikzlibrary{calc,matrix,arrows,decorations.pathmorphing}
\usepackage{calc}
\usepackage[cal=boondox,scr=boondoxo]{mathalfa}
\usepackage{marginnote}
\usepackage{booktabs}
\usepackage{scalerel}[2016/12/29]
\usepackage{enumitem}
\usepackage{epsfig}
\usepackage{placeins}

\newcommand\be{\begin{equation}}
\newcommand\ee{\end{equation}}

\newtheorem{theorem}{Theorem}[section]
\newtheorem{corollary}[theorem]{Corollary}
\newtheorem{proposition}[theorem]{Proposition}
\newtheorem{lemma}[theorem]{Lemma}

\theoremstyle{definition}
\newtheorem{definition}[theorem]{Definition}

\theoremstyle{remark}
\newtheorem{remark}[theorem]{Remark}

\DeclareMathSymbol{\widetildesym}{\mathord}{largesymbols}{"65}

\DeclareMathOperator{\Irr}{Irr}

\newcommand{\Z}{\mathbb{Z}}

\newcommand{\R}{\mathbb{R}}

\newcommand{\rmt}{\mathrm{t}}

\newcommand{\rmL}{\mathrm{L}}

\newcommand{\rmV}{\mathrm{V}}
\newcommand{\rmX}{\mathrm{X}}

\newcommand{\bbA}{\mathbb{A}}
\newcommand{\bbB}{\mathbb{B}}

\newcommand{\bbD}{\mathbb{D}}

\newcommand{\bbM}{\mathbb{M}}

\newcommand{\bbS}{\mathbb{S}}

\DeclareRobustCommand{\bbSigma}{\mathbin{\text{\includegraphics[height=\heightof{$\mathbf{\Sigma}$}]{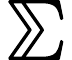}}}}

\newcommand{\calC}{\mathcal{C}}
\newcommand{\calD}{\mathcal{D}}

\newcommand{\calI}{\mathcal{I}}

\newcommand{\calR}{\mathcal{R}}

\newcommand{\calV}{\mathcal{V}}
\newcommand{\calX}{\mathcal{X}}

\newcommand{\frakg}{\mathfrak{g}}

\renewcommand{\epsilon}{\varepsilon}
\renewcommand{\theta}{\vartheta}
\renewcommand{\phi}{\varphi}
\renewcommand{\Gamma}{\varGamma}
\renewcommand{\Sigma}{\varSigma}

\newcommand{\id}{\mathrm{id}}

\newcommand{\tr}{\mathrm{tr}}
\newcommand{\lptr}{\mathrm{tr}_{\mathrm{L}}}
\newcommand{\rptr}{\mathrm{tr}_{\mathrm{R}}}

\newcommand{\lev}{\smash{\stackrel{\leftarrow}{\mathrm{ev}}}}
\newcommand{\lcoev}{\smash{\stackrel{\longleftarrow}{\mathrm{coev}}}}
\newcommand{\rev}{\smash{\stackrel{\rightarrow}{\mathrm{ev}}}}
\newcommand{\rcoev}{\smash{\stackrel{\longrightarrow}{\mathrm{coev}}}}

\DeclareMathOperator{\disjun}{\sqcup}
\DeclareMathOperator{\din}{\dot{\Rightarrow}}

\renewcommand{\leq}{\leqslant}
\renewcommand{\geq}{\geqslant}
\newcommand{\pc}{\mathrm{pc}}

\newcommand{\mods}[1]{\operatorname{\mathnormal{#1}-mod}}

\newcommand{\cat}{\mathcal{C}}
\newcommand{\brk}[1]{{{\left\langle{#1}\right\rangle}}}

\renewcommand{\sl}{\mathfrak{sl}}
\newcommand{\SL}{\mathrm{SL}}

\newcommand{\End}{\mathrm{End}}

\newcommand{\Vect}{\mathrm{Vect}}

\newcommand{\op}{\mathrm{op}}

\DeclareRobustCommand{\one}{\mathbin{\text{\includegraphics[height=\heightof{$\mathbf{1}$}]{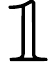}}}}

\newcommand{\Cob}{\mathrm{Cob}}
\newcommand{\adCob}{\check{\mathrm{C}}\mathrm{ob}}

\newcommand{\Proj}{\mathrm{Proj}}

\renewenvironment{quote}{
   \list{}{
     \leftmargin\parindent
     \rightmargin\leftmargin
   }
   \item\relax
}
{\endlist}

\makeatletter
\newcommand{\subalign}[1]{
  \vcenter{
    \Let@ \restore@math@cr \default@tag
    \baselineskip\fontdimen10 \scriptfont\tw@
    \advance\baselineskip\fontdimen12 \scriptfont\tw@
    \lineskip\thr@@\fontdimen8 \scriptfont\thr@@
    \lineskiplimit\lineskip
    \ialign{\hfil$\m@th\scriptstyle##$&$\m@th\scriptstyle{}##$\crcr
      #1\crcr
    }
  }
}
\makeatother

\def\clap#1{\hbox to 0pt{\hss#1\hss}}

\newcommand{\brC}{c}

\newcommand{\Tideal}{\mathcal{I}}
\newcommand{\trI}{\rmt}

\newcommand{\eend}{\mathcal{E}} 
\newcommand{\coend}{\mathcal{L}} 

\newcommand{\multL}{\mu}
\newcommand{\unitL}{\eta}
\newcommand{\copL}{\Delta}
\newcommand{\counitL}{\epsilon}
\newcommand{\antipL}{S}

\newcommand{\intL}{\Lambda}
\newcommand{\cointL}{\Lambda^{\mathrm{co}}}
\newcommand{\pairL}{\omega}

\newcommand{\DD}{D}
\newcommand{\T}{\chi}

\newcommand{\ok}{{\ensuremath{\Bbbk}}}

\newcommand{\modT}{\mathcal{T}}
  
\newcommand{\pic}[2][0]{\raisebox{-0.5\height + 2.5pt + #1pt}{\includegraphics{#2.pdf}}}  
  
\newcommand{\arxiv}[2]{\href{http://arXiv.org/abs/#1}{\texttt{arXiv:\allowbreak #1} #2}}
\newcommand{\doi}[2]{\href{http://doi.org/#1}{#2}}

\begin{document}

\raggedbottom

\title{3-Dimensional TQFTs From Non-Semisimple Modular Categories}

\author[De Renzi]{Marco De Renzi} 
\address{Department of Mathematics, Faculty of Science and Engineering, Waseda University, 3-4-1 \={O}kubo, Shinjuku-ku, Tokyo, 169-8555, Japan} 
\email{m.derenzi@kurenai.waseda.jp}
\address{Institute of Mathematics, University of Zurich, Winterthurerstrasse 190, CH-8057 Zurich, Switzerland}
\email{marco.derenzi@math.uzh.ch}

\author[Gainutdinov]{Azat M. Gainutdinov}
\address{Institut Denis Poisson, CNRS, Universit\'e de Tours, Universit\'e d'Orl\'eans, Parc de Grandmont, 37200 Tours, France}
\email{azat.gainutdinov@lmpt.univ-tours.fr}

\author[Geer]{Nathan Geer}
\address{Mathematics \& Statistics, Utah State University, Logan, Utah 84322, USA} \email{nathan.geer@gmail.com}

\author[Patureau-Mirand]{Bertrand Patureau-Mirand}
\address{Univ. Bretagne - Sud, UMR 6205, LMBA, F-56000 Vannes, France}
\email{bertrand.patureau@univ-ubs.fr}

\author[Runkel]{Ingo Runkel}
\address{Fachbereich Mathematik, Universit\"at Hamburg,
Bundesstra\ss e 55, 20146 Hamburg, Germany}
\email{ingo.runkel@uni-hamburg.de}

\begin{abstract}
 We use modified traces to renormalize Lyubashenko's closed 3-man\-i\-fold invariants coming from twist non-de\-gen\-er\-ate finite unimodular ribbon categories. Our construction produces new topological invariants which we upgrade to 2+1-TQFTs under the additional assumption of factorizability. The resulting functors provide monoidal extensions of Lyubashenko's mapping class group representations, as discussed in \cite{DGGPR20}. This general framework encompasses important examples of non-semisimple modular categories coming from the representation theory of quasi-Hopf algebras, which were left out of previous non-semisimple TQFT constructions.
\end{abstract}

\maketitle
\setcounter{tocdepth}{2}

\tableofcontents

\date{\today}

\section{Introduction}

In this paper we show how to construct a topological invariant of closed 3-dimensional manifolds out of any finite unimodular ribbon category $\calC$ satisfying a weak non-degeneracy condition, and how to extend it to a $2+1$-dimensional \textit{Topological Quantum Field Theory} (\textit{TQFT} for short) in case $\calC$ is also factorizable. Our results generalize several previous constructions, from Reshetikhin-Turaev TQFTs \cite{T94}, which we recover when $\calC$ is semisimple, to Lyubashenko's mapping class group representations \cite{L94}, as we show in \cite{DGGPR20}, to the family of non-semisimple TQFTs constructed in \cite{DGP17} using finite-dimensional factorizable ribbon Hopf algebras. The advantage of our new approach is that it is applicable to the case of weak Hopf algebras and quasi-Hopf algebras. For example, some important categories, such as those coming from the representation theory of quantum $\sl_2$ at even roots of unity \cite{CGR17}, did not fit in the previous framework, and were therefore not eligible for a TQFT construction up to now.

Let us state our main result. First, recall that, paraphrasing Atiyah \cite{A88}, a $2+1$-TQFT can be defined as a symmetric monoidal functor from a category of cobordisms of dimension $2+1$ to a category of vector spaces over a field $\Bbbk$. The use of indefinite articles in this definition is motivated by the fact that both cobordisms and vector spaces are usually allowed to carry additional structures, which can vary according to the specific construction, and it is customary to refer to all such functors as TQFTs. In this paper, the relevant structure in the definition of our source category depends on the choice of a finite ribbon category $\calC$. Very roughly speaking, it consists in decorations given by special sets of oriented vertices labeled with objects of $\calC$ embedded into surfaces, and by special oriented graphs labeled with objects and morphisms of $\calC$ embedded into cobordisms. The crucial property of these decorations is a certain \textit{admissibility condition} whose goal is to ensure that every connected component of every closed cobordism in our source category contains at least one projective object of $\calC$ among the labels of its embedded graph. This results in the definition of the \textit{admissible cobordism category} $\adCob_\calC$ of Section \ref{S:admissible_cobordisms}. One important difference between $\adCob_\calC$ and usual cobordism categories is that $\adCob_\calC$ is \emph{not} rigid, unless $\calC$ is semisimple. It is actually this property that, under suitable hypotheses for $\calC$, allows us to use $\adCob_\calC$ as the domain of our TQFT\footnote{It is proved in \cite{BDSV15} that a $2+1$-TQFT which extends to a $1+1+1$-ETQFT defined over the whole rigid $2$-category of cobordisms needs to have a semisimple circle category, while the one we construct here determines a non-semisimple circle category \cite{D21}.}. More precisely, we employ the term \textit{modular category} in the non-semisimple sense to denote a finite factorizable ribbon category. Then, let us use the notation $\calC(V,W)$ for the vector space of morphisms from $V \in \calC$ to $W \in \calC$, and let us write $\coend = \int^{X \in \calC} X^* \otimes X \in\calC$ for the \textit{coend} in $\calC$, see Section \ref{S:coend}. We show in Theorem~\ref{T:monoidality} and Proposition~\ref{P:state-space-iso}:

\begin{theorem}\label{T:main}
 If $\calC$ is a modular category over an algebraically closed field $\Bbbk$, then there exists a $2+1$-TQFT $\rmV_\calC : \adCob_\calC \to \Vect_\Bbbk$ mapping every closed surface of genus $g$ decorated with $n$ positive vertices labeled by $V_1, \ldots, V_n \in \calC$ to a vector space isomorphic to the linear dual\footnote{The isomorphism with $\calC(\coend^{\otimes g} \otimes V_1 \otimes \ldots \otimes V_n,\one)$ depends on the choice of a basis, while the one with its linear dual does not.} of
 \[
  \calC(\coend^{\otimes g} \otimes V_1 \otimes \ldots \otimes V_n,\one).
 \]
\end{theorem}

\subsection{Previous results}

In order to put our work into perspective, let us recall a brief history of non-semisimple quantum topology. At the beginning of the '90s, Hennings constructed the first family of so-called \textit{non-semisimple} quantum invariants of closed 3-manifolds \cite{H96}. The main algebraic tool used in the process is a finite-di\-men\-sion\-al ribbon Hopf algebra $H$. The construction requires certain algebraic conditions of $H$, namely \textit{unimodularity} and \textit{twist non-degeneracy}, but it does not require semisimplicity, hence the name. When $H$ is semisimple, his results recover the Reshetikhin-Turaev invariant of \cite{T94} associated with the category of finite-dimensional representations $\mods{H}$. In 1994 Lyubashenko generalized Hennings' construction, and simultaneously obtained representations of mapping class groups of surfaces \cite{L94}. In his approach, the Hopf algebra $H$ and its category of representations $\mods{H}$ are replaced with a general finite ribbon category $\calC$. This time, in order to obtain mapping class group representations, the requirement is that $\calC$ should be \textit{factorizable}. This means that the natural Hopf pairing defined on the coend $\coend \in \calC$ has to be non-degenerate, or equivalently that the only transparent objects of $\calC$ have to be direct sums of the tensor unit \cite{S16}. When the category $\calC$ is of the form $\mods{H}$ for some finite-dimensional factorizable ribbon Hopf algebra $H$, then Lyubashenko's invariant recovers Hennings' one. In 2001 Kerler and Lyubashenko constructed 2-functorial extensions of these mapping class group representations, but only for 3-dimensional cobordisms with corners between connected surfaces with boundary \cite{KL01}. Indeed, they found deep obstructions when trying to translate the monoidal structure induced by disjoint union of surfaces and cobordisms into the one induced by tensor product of vector spaces and linear maps. More precisely, if $\calC$ is non-semisimple, then its corresponding Lyubashenko invariant vanishes against all closed 3-manifolds whose first Betti number is strictly positive \cite{O95,K96b}. This means a TQFT extending Lyubashenko's invariant would have to assign 0-dimensional vector spaces to every closed surface, which is contradictory.

These difficulties were recently overcome through the use of so-called \textit{modified traces}, whose theory was developed in \cite{GPT07,GKP10,GKP11,GPV11,BBG18,GKP18}. These techniques were first used in a different, not necessarily finite setting for the construction of certain non-semisimple quantum invariants of closed 3-dimensional manifolds known as \textit{CGP} invariants \cite{CGP12}. These invariants have later been upgraded to $2+1$-dimensional TQFTs, at first only for the so-called \textit{unrolled} version of the quantum group of $\sl_2$ at roots of unity \cite{BCGP14}, and then in general, and even for higher categorical analogues called $1+1+1$-dimensional \textit{Extended TQFTs} (\textit{ETQFTs} for short) \cite{D17}. The main algebraic ingredient for these constructions is provided by \textit{relative modular categories}, which are (not necessarily semisimple) ribbon categories featuring a possibly infinite number of isomorphism classes of simple objects. The resulting quantum invariants and (E)TQFTs are defined for manifolds and cobordisms decorated with 1-dimensional cohomology classes, and symmetric monoidality holds in a graded sense. This theory generalizes the standard approach of Reshetikhin-Turaev in a different direction with respect to the one of Kerler-Lyubasheko, as the intersection between relative modular categories and finite factorizable ribbon categories is limited to semisimple modular categories.

More recently, modified traces were also used to \textit{renormalize} Hennings' construction, at first only for the \textit{restricted} quantum group of $\sl_2$ at roots of unity \cite{BBG17}, and then for general twist non-degenerate finite-dimensional unimodular ribbon Hopf algebras $H$ \cite{DGP17}. The resulting quantum invariants of closed 3-manifolds contain some classical ones, such as Kashaev's invariants of knots \cite{K96} and their generalized versions \cite{M13}. Renormalized Hennings invariants are profoundly different from the original ones, as they extend to fully monoidal TQFTs whenever $H$ is factorizable. They have been shown to coincide with CGP invariants associated with the trivial cohomology class in the case of quantum groups at roots of unity of odd order \cite{DGP18}. However, all these constructions were performed in the special framework of Hopf algebras, and it is natural to wonder whether this restriction is necessary. This work provides the first step towards a categorical formulation of the constructions above.

\subsection{Summary of the construction}\label{S:summary}

In this paper, we use modified traces for arbitrary finite ribbon categories in order to renormalize Lyubashenko's 3-manifold invariants, and to extend them to $2+1$-TQFTs. The construction is divided into three main parts:
\begin{enumerate}
 \item In Section \ref{S:categories}, we recall the algebraic setup needed for our topological constructions. This includes definitions of modified traces and coends, together with their structure and properties, as well as some important consequences of unimodularity and factorizability. 
 \item In Section \ref{S:3-manifold_invariants}, we first define a monoidal functor with target an arbitrary finite unimodular ribbon category $\calC$, and source the category of so-called bichrome graphs. Next, we proceed to use this functor in order to construct closed 3-man\-i\-fold invariants under the additional assumption of twist non-de\-gen\-er\-a\-cy of $\calC$, by suitably combining Lyubashenko's work with the theory of modified traces. Our construction is actually parameterized by modified traces on tensor ideals, meaning every non-zero modified trace $\rmt$ on a tensor ideal $\Tideal \subseteq \calC$ determines a topological invariant.
 \item In Section \ref{S:TQFTs}, we extend these 3-manifold invariants to $2+1$-TQFTs under the additional assumption of factorizability of $\calC$. Our approach requires the use of the tensor ideal $\Proj(\calC)$ of projective objects of $\calC$, which supports a unique non-zero modified trace up to scalar.
\end{enumerate}

Let us quickly outline the construction. We start by considering a finite unimodular ribbon category $\calC$ (see Section \ref{S:unimodularity}). When $\calC$ is twist non-degenerate (see Section \ref{S:non-degeneracy}), it provides the basic ingredient for Lyubashenko's 3-manifold invariant $\rmL_\calC$. The construction crucially exploits the universal property of the coend $\coend \in \calC$, which has the following topological significance: every isotopy class of \textit{$\ell$-bottom tangles} (see Section~\ref{SS:LRT_functor}) determines a \textit{unique} morphism in $\calC$ from $\coend^{\otimes \ell}$ to $\one$.

In the same spirit as in \cite{BBG17,DGP17}, the idea is to add modified traces to our toolbox for the construction. In order to do this, we need to work with a mild generalization of ribbon graphs called \textit{bichrome graphs}, which have edges of two kinds: red and blue. While blue edges are labeled as usual with objects of $\calC$, red edges are unlabeled, and they play a different role in the construction. Indeed, they should be treated as portions of surgery presentations of closed 3-manifolds. This means they should be evaluated using a special morphism $\intL \in \calC(\one,\coend)$ called the \textit{integral} of the coend $\coend \in \calC$. All this is made precise by the construction of the \textit{Lyubashenko-Reshetikhin-Turaev} functor $F_\intL : \calR_\intL \to \calC$ with  source the category of bichrome graphs $\calR_\intL$, to which we devote Section \ref{SS:LRT_functor}. The advantage is that the blue part of closed bichrome graphs can now be used to incorporate modified traces in the construction. Indeed, the standard categorical trace is degenerate in the non-semisimple case, and often vanishes on proper tensor ideals $\Tideal \subset \calC$, thus leading to trivial topological invariants. On the other hand, many of these ideals admit non-degenerate modified traces $\rmt$ (which in general are not unique). For instance, in our setting, $\calI = \Proj(\calC)$ is always a possible choice, in which case $\rmt$ always exists (and is unique up to scalar). However, if we want to use $\rmt$ as a tool for extracting topological information out of a bichrome graph $T$, we need to assume that there exists a blue edge of $T$ labeled by an object of $\calI$. This leads to the definition of \textit{admissible bichrome graphs}, which can be fed to a \textit{renormalized} invariant $F'_{\intL,\rmt}$ obtained by combining the functor $F_\intL$ with the modified trace $\rmt$. Up to rescaling $F'_{\intL,\rmt}$ using so-called \textit{stabilization coefficients}, an operation which requires twist non-degeneracy of $\calC$, we obtain a topological invariant of closed 3-manifolds decorated with admissible bichrome graphs. More precisely, let $M$ be a closed 3-manifold, let $T$ be an admissible bichrome graph embedded into $M$, and let $L$ be a surgery presentation of $M$, which we interpret as a red framed link in $S^3$ with $\ell$ components and signature~$\sigma(L)$. We define the renormalized Lyubashenko invariant as
\[
 \rmL'_{\calC,\calI}(M,T) := \calD^{-1-\ell} \delta^{-\sigma(L)} F'_{\intL,\rmt}(L \cup T).
\]
where the coefficients $\calD$ and $\delta$ are related to stabilization coefficients as explained in Section \ref{SS:3-manifold_invariants}. We show in Theorem \ref{T:admissible_closed_3-manifold_invariant} that this is indeed an invariant of the pair $(M,T)$. When $\calC$ is semisimple its only non-zero ideal is $\calI = \calC$, its only non-zero trace, up to scalar, is $\rmt = \tr_\calC$, and $\rmL'_{\calC,\calI}$ recovers the standard Reshetikhin-Turaev invariant. When $\calC$ is the representation category of a twist non-degenerate finite-dimensional unimodular ribbon Hopf algebra, then $\rmL'_{\calC,\calI}$ generalizes the renormalized Hennings invariant of \cite{DGP17}, whose construction was performed in the special case $\calI = \Proj(\calC)$.

The second construction of this paper requires a framework which is slightly more rigid. Indeed, in order to extend the renormalized Lyubashenko invariant to a TQFT $\rmV_\calC : \adCob_\calC \to \Vect_\Bbbk$, we need to assume $\calC$ is also factorizable\footnote{The assumption is used in Lemma~\ref{L:cutting}, and is crucial for proving monoidality of the TQFT.}, and we also need to pick the ideal $\calI = \Proj(\calC)$ as the domain of our modified trace~$\rmt$. Then, the universal construction of \cite{BHMV95} provides a general procedure for the definition of a functorial extension of $\rmL'_\calC := \smash{\rmL'_{\calC,\Proj(\calC)}}$. Just like in \cite{DGP17}, we work with a category $\adCob_\calC$ of admissible cobordisms. Decorations for objects are essentially provided by embedded sets of blue marked points carrying framings, orientations, and labels, which is what we get when we intersect transversely a surface with a bichrome graph inside a closed 3-manifold, if we avoid all red edges. Consequently, decorations for morphisms are essentially provided by bichrome graphs properly embedded into cobordisms, although not by arbitrary ones. Indeed, we require the presence of a blue edge labeled by a projective object of $\calC$, but only for connected components which are disjoint from the incoming boundary. This means for example that coevaluation morphisms for surfaces without points labeled by projective objects of $\calC$ are not allowed in our construction, while evaluation morphisms are. With this definition in place, we can apply the universal construction of \cite{BHMV95}, and consider for every decorated surface $\bbSigma \in \adCob_\calC$ vector spaces $\calV(\bbSigma)$ and $\calV'(\bbSigma)$ freely generated by cobordisms of the form $\bbM_{\bbSigma} : \varnothing \to \bbSigma$ and of the form $\bbM'_{\bbSigma} : \bbSigma \to \varnothing$, respectively. The state space $\rmV_\calC(\bbSigma)$ of $\bbSigma$ is then defined as the quotient of $\calV(\bbSigma)$ with respect to the radical of the bilinear form $\langle \cdot,\cdot \rangle_{\bbSigma} :  \calV'(\bbSigma) \times \calV(\bbSigma) \to \Bbbk$ determined by
\[
 \langle \bbM'_{\bbSigma},\bbM_{\bbSigma} \rangle_{\bbSigma} := \rmL'_\calC(\bbM'_{\bbSigma} \circ \bbM_{\bbSigma}).
\]
This automatically induces a functor $\rmV_\calC : \adCob_\calC \to \Vect_\Bbbk$. The remaining step is to establish monoidality on objects, the hardest part of the construction, which is done in Theorem \ref{T:monoidality}. Finally, in Section \ref{S:identification_state_spaces} we prove that state spaces obtained from the above quotient construction agree with morphism spaces in $\calC$ as stated in Theorem \ref{T:main}. We show in \cite{DGGPR20} that the mapping class group representations induced by $\rmV_\calC$ on these spaces are isomorphic to the ones introduced by Lyubashenko in \cite{L94}. Furthermore, in the case of the small quantum group of $\sl_2$, we establish an interesting property of our TQFT: the action of any Dehn twist on the state space of a surface with empty decorations has infinite order, something which never happens for Reshetikhin-Turaev TQFTs.

When $\calC = \mods{H}$ for a finite-dimensional factorizable ribbon Hopf algebra $H$, our construction is equivalent to the one of \cite{DGP17}, as proved in \cite[Sec.~2.4~\&~App.~C]{DGGPR20}, although there are some slight differences: our notion of bichrome graph is simpler, and certain proofs are more straightforward due to the universal property of the coend. Another advantage is that our construction allows now $H$ to be a (weak) quasi-Hopf algebra, as long as it is still finite-dimensional, ribbon, and factorizable. It would be particularly interesting to study the case of the quasi-Hopf version of the restricted quantum group of $\sl_2$ at even roots of unity \cite{GR15, CGR17}, of the symplectic fermion quasi-Hopf algebra \cite{FGR17b}, and of the small quasi-quantum group of a general simple Lie algebra $\frakg$ \cite{GLO18, N18}, since the corresponding Hopf versions are not quasi-triangular.

\addtocontents{toc}{\protect\setcounter{tocdepth}{0}}

\subsection*{Conventions and notations}

Throughout this paper, we fix an al\-ge\-bra\-i\-cal\-ly closed field $\Bbbk$. The terms \textit{linear}, \textit{vector space}, and \textit{algebra} will always be used as shorthand for \textit{$\Bbbk$-linear}, \textit{$\Bbbk$-vector space}, and \textit{$\Bbbk$-algebra} respectively. We denote by $\Vect_\Bbbk$ the linear category of vector spaces, and, if $A$ is an algebra, we denote with $\mods{A}$ the linear category of finite-dimensional left $A$-modules.

\addtocontents{toc}{\protect\setcounter{tocdepth}{2}}

\section{Unimodular, twist non-degenerate, and modular categories}\label{S:categories}

In this section we collect definitions and results related to ribbon categories that we will need for our construction. In order to keep notation light, we appeal to a few coherence results. Indeed, thanks to \cite[Thm.~XI.3.1]{M71}, every monoidal category is equivalent, as a monoidal category, to a strict one. Furthermore, thanks to \cite[Thm.~2.2]{NS05}, every pivotal category is equivalent, as a pivotal category, to a strict one. Therefore, throughout the whole section, we make the following assumption:
\begin{quote}
 $\cat$ is a finite ribbon category whose underlying pivotal category is strict\footnote{When $\calC$ is not a strict pivotal or even a strict monoidal category, it is easy to insert coherence isomorphisms in all equations appearing throughout this section.}.
\end{quote}

\subsection{Finite ribbon categories}

Following \cite{EGNO15}, a linear category is {\em finite} if it is equivalent, as a linear category, to $\mods{A}$ for some finite-dimensional algebra $A$. In particular, a finite linear category is abelian. By a {\em finite ribbon category} we mean a finite linear category which is in addition a ribbon category such that the tensor product $\otimes$ is bilinear and the tensor unit $\one$ is simple. Equivalently, in the language of \cite{EGNO15}, $\cat$ is a finite tensor category which is in addition ribbon.

Our conventions for structure morphisms of $\cat$ are as follows. Every object $V$ in $\cat$ has a two-sided dual $V^\ast$, and we denote left and right duality morphisms by
\begin{align*}
 \lev_V &: V^\ast \otimes V \to \one, &
 \lcoev_V &: \one \to V \otimes V^\ast, \\*
 \rev_V &: V \otimes V^\ast \to \one, &
 \rcoev_V &: \one \to V^\ast \otimes V.
\end{align*}
The natural families of isomorphisms defining braiding and twist are denoted by
\[
 \brC_{V,W} : V \otimes W \to W \otimes V, \qquad 
 \theta_V : V \to V.
\]
By definition, $\theta$ is a twist if and only if, for all $V,W \in \cat$, we have
\begin{equation}\label{eq:theta-braiding-prop}
 \theta_{V \otimes W}  = \brC_{W,V} \circ \brC_{V,W} \circ (\theta_V \otimes \theta_W), \qquad
 \theta_{V^*} = (\theta_V)^*.
\end{equation}
These structural morphisms are graphically represented by
\[
 \includegraphics{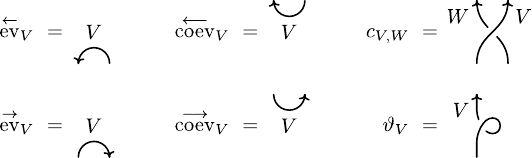}
\] 
Note that we read graphical representations of morphisms from bottom to top, interpreting upward and downward oriented strands as identity morphisms of their labels and of their duals respectively, with tensor product given by juxtaposition.

\subsection{Projective objects and unimodularity}\label{S:unimodularity}

From now on, $\cat$ is a finite ribbon category. We write $\Irr \subset \cat$ for a choice of a set of representatives of isomorphism classes of simple objects of $\cat$, and we assume $\one \in \Irr$. Since $\cat$ is finite, $\Irr$ is a finite set. Projective covers exist in $\cat$, and we denote by $P_V$ the projective cover of $V \in \Irr$. Any object of the form
\begin{equation}\label{E:proj_gen}
 G = \bigoplus_{V \in \Irr} P_V^{\oplus n_V},
\end{equation}
where all $n_V$ are strictly positive, is a projective generator of $\cat$.

We denote with $\Proj(\cat) \subseteq \cat$ the full subcategory of projective objects in $\cat$. Since $\cat$ is rigid, the tensor product $\otimes$ is exact, and for $P \in \Proj(\cat)$ and $V \in \cat$ it follows that $P \otimes V$ is again projective, see \cite[Sec.~4.2]{EGNO15}. Furthermore, direct summands of projective objects are projective.

An important role in our construction is played by the projective cover $P_{\one}$ of the tensor unit $\one$, together with its canonical surjection $\epsilon_{\one} \colon P_{\one} \to \one$. Note that $P_{\one}$ is simple if and only if $P_{\one} \cong \one$, and that in this case it follows that $\Proj(\cat) = \cat$. Thus, $\cat$ is semisimple if and only if $P_{\one}$ is simple.

A finite tensor category is called \textit{unimodular} if $P^*_{\one} \cong P_{\one}$.

\subsection{Tensor ideals and traces}

Here we recall from \cite{GKP10} the notion of a modified trace on a tensor ideal in $\cat$. By defintion, a {\em tensor ideal} $\Tideal \subseteq \calC$ is a full subcategory which is closed under retracts (i.e. taking direct summands) and such that for all $X \in \Tideal$ and $V \in \cat$ we have $X \otimes V \in \Tideal$. Since closure under retracts implies repleteness (i.e. closure under isomorphisms), and since $\cat$ is braided, we automatically have $V \otimes X \in \Tideal$ too.

The left and right partial traces of an endomorphism $f \in \End_\calC(V \otimes W)$ are the endomorphisms $\lptr(f) \in \End_\calC(W)$ and $\rptr(f) \in \End_\calC(V)$ defined as
\[
 \raisebox{-0.5\height}{\includegraphics{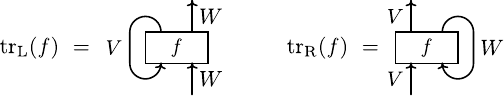}}
\]

A {\em trace $\trI$ on a tensor ideal $\Tideal \subseteq \calC$} is a family of linear maps
\[
 \{ \trI_X : \End_\calC(X) \to \ok \}_{X \in \Tideal}
\]
subject to the following conditions:
\begin{itemize}
	\item[1)] {\em Cyclicity}: For all $X,Y \in \Tideal$ and $f : X \to Y$, $g : Y \to X$ we have 
	 \[
	  \trI_Y(f \circ g) = \trI_X(g \circ f);
	 \]
	\item[2R)] {\em Right partial trace}: For all $X \in \Tideal$, $V \in \cat$ and
	 $h \in \End_\calC(X \otimes V)$, 
	 \[
	  \trI_{X \otimes V}(h) = \trI_X(\rptr(h));
	 \]
	\item[2L)] {\em Left partial trace}: For all $X \in \Tideal$, $V \in \cat$ and $h \in \End_\calC(V \otimes X)$, 
	 \[
	  \trI_{V \otimes X}(h) = \trI_X(\lptr(h)).
	 \]
\end{itemize}
Since $\cat$ is ribbon, conditions 2R) and 2L) above are equivalent \cite{GKP10}.

We say a trace $\trI$ on an ideal $\Tideal \subseteq \calC$ is \textit{non-degenerate} if for every $V \in \Tideal$ and every $W \in \calC$ the pairing $\trI_V( \cdot \circ \cdot ) : \calC(W,V) \times \calC(V,W) \to \Bbbk$ is non-degenerate. An important example of a tensor ideal is the projective ideal $\Proj(\cat)$. It is shown in \cite[Thm.~5.5~\&~Cor.~5.6]{GKP18} that:

\begin{proposition}\label{P:uniqueness_of_trace}
 If $\cat$ is also unimodular, then there exists a unique-up-to-scalar non-zero trace $\trI$ on $\Proj(\cat)$, and furthermore $\trI$ is non-de\-gen\-er\-ate.
\end{proposition}

\subsection{Coends and ends}\label{S:coend}

We will now recall some well-known facts about the end of the functor $\calC \times \calC^{\op}\to \calC$ sending every $(U,V) \in \calC \times \calC^{\op}$ to $U \otimes V^* \in \calC$ and about the coend of the functor $\calC^{\op} \times \calC\to \calC$ sending every $(U,V) \in \calC$ to $U^* \otimes V \in \calC$. We use the notation
\begin{align*}
 &\eend := \int_{X\in\calC} X\otimes X^*, & &\coend := \int^{X\in\calC} X^*\otimes X, \\
 &j_X \colon \eend \to X\otimes X^*, & &i_X \colon X^*\otimes X \to \coend,
\end{align*}
for the end and the coend respectively, and for their corresponding dinatural transformations. See \cite[Sec.~IX.4--IX.6]{M71} for a definition of dinatural transformations, ends, and coends, and see \cite[Sec.~4]{FS10} or \cite[Sec.~3]{FGR17a} for the specific coend $\coend$. In \cite{L94}, $\coend$ was used as a key ingredient for the construction of representations of mapping class groups of surfaces in certain morphism spaces of $\calC$ that will appear in Section \ref{S:pairing} too.

The coend $\coend$ carries the structure of a Hopf algebra in $\cat$ \cite{M93,L95} (the same holds for the end $\eend$, but we will not need it). For our conventions on braided Hopf algebras, as well as for a review of the construction of the Hopf algebra structure on $\coend$, we refer to \cite{FGR17a}. Our notation for structure morphisms of $\coend$ is
\begin{align*}
\text{(Product)} \quad & \multL : \coend \otimes \coend \to \coend, &
\text{(Unit)} \quad & \unitL : \one \to \coend, \\*
\text{(Coproduct)} \quad & \copL : \coend \to \coend \otimes \coend, &
\text{(Counit)} \quad & \counitL : \coend \to \one, \\*
\text{(Antipode)} \quad & \antipL : \coend \to \coend.  
\end{align*}
All these maps are determined by the universal property of $\coend$. Indeed, they are uniquely defined by 
\begin{align}
 \raisebox{-0.5\height}{\includegraphics{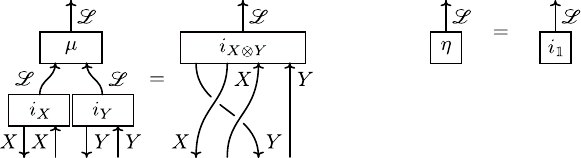}} \label{E:algebra_structure} \\*[0.5\baselineskip]
 \raisebox{-0.5\height}{\includegraphics{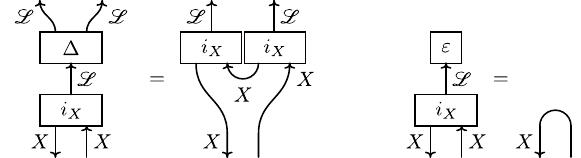}} \label{E:coalgebra_structure} \\*[0.5\baselineskip]
 \raisebox{-0.5\height}{\includegraphics{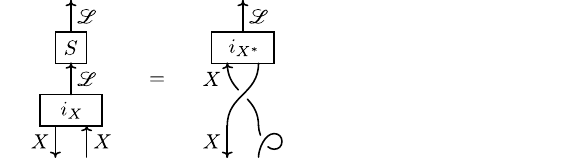}} \label{E:antipode}
\end{align}
Note that when $\calC$ is not a strict pivotal category, canonical isomorphisms $Y^* \otimes X^* \cong (X \otimes Y)^*$ for the product, $\one \cong\one^* \otimes \one$ for the unit, and $X \cong X^{**}$ for the antipode are needed. See \cite[Sec.~3.3]{FGR17a} for more details.

The coend is equipped with a Hopf pairing $\pairL : \coend \otimes \coend \to \one$ that will be important below, which is uniquely defined via the universal property of $\coend$ by
\begin{equation}\label{E:Hopf_pairing}
 \raisebox{-0.5\height}{\includegraphics{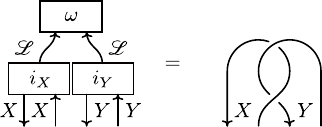}}
\end{equation}
Recall that $\pairL$ being a Hopf pairing means in particular it satisfies
\begin{align}
 \pairL \circ (\id_\coend \otimes \multL) &= \pairL \circ (\id_\coend \otimes \pairL \otimes \id_\coend) \circ (\copL \otimes  \id_{\coend \otimes \coend}), \label{E:Hopf_1} \\*
 \pairL \circ (\antipL \otimes \id_\coend) &= \pairL \circ (\id_\coend \otimes \antipL), \label{E:Hopf_2} 
\end{align}
as shown in \cite[Thm.~3.7]{L95} and in \cite[Eq.~(5.2.8)]{KL01}. The pairing $\tilde{\pairL} : \coend \otimes \coend \to \one$ given by $\pairL \circ (\antipL \otimes \id_\coend) = \pairL \circ (\id_\coend \otimes \antipL)$ satisfies
\begin{equation}\label{E:mirrored_pairing}
 \raisebox{-0.5\height}{\includegraphics{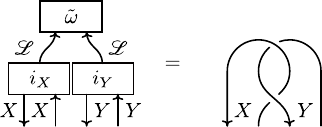}}
\end{equation}
with braidings replaced by inverse ones. Moreover, \cite[Lem.~5.2.4]{KL01} gives
\begin{equation}
 \pairL \circ c_{\coend,\coend}^{-1} = \pairL \circ (\antipL \otimes \antipL). \label{E:Hopf_3}
\end{equation}

\subsection{Integrals and cointegrals}

Let us assume that $\cat$ is in addition unimodular. A morphism $\intL \in \calC(\one,\coend)$ is called a \textit{right integral of $\coend$} if it satisfies
\begin{equation}\label{eq:right-integral-def}
 \multL \circ (\intL \otimes \id_\coend)=  \intL \circ \counitL.
\end{equation}
A left integral of $\coend$ is defined similarly\footnote{In the non-unimodular case, a right/left integral would be a map out of the distinguished invertible object, not out of the tensor unit, see e.g.~\cite{KL01, BGR20}}. It is known that right/left integrals of $\coend$ exist and are unique up to scalar, see \cite[Prop.~4.2.4]{KL01}. Furthermore, as we are in the unimodular case, each left integral is also a right integral and vice versa, see \cite[Thm.~6.9]{S14}. In other words, integrals are two-sided.

Dually, we have left and right cointegrals of $\coend$. A morphism $\cointL \in \calC(\coend,\one)$ is called a \textit{right cointegral of $\coend$} if it satisfies
\begin{equation}
 (\cointL \otimes \id_\coend) \circ \copL= \unitL \circ  \cointL,
\end{equation}
and similarly for the left version. Just like for integrals, since we are in the unimodular case, we have the following.

\begin{lemma}
 Cointegrals of $\coend$ are two-sided.
\end{lemma}

\begin{proof}
 The lemma is an immediate consequence of the following more general claim: for any morphism $f : \coend \to \one$ we have
 \begin{equation}\label{eq:coint-twosided-aux1}
  (\id_\coend \otimes f) \circ \copL = (f \otimes \id_\coend) \circ \copL.
 \end{equation}	
 To show this identity, compose both sides with $i_X$ from the right, and substitute the defining property of the coproduct $\copL$ given by Equation \eqref{E:coalgebra_structure}. Now Equation \eqref{eq:coint-twosided-aux1} is equivalent to 
 \[
  \includegraphics{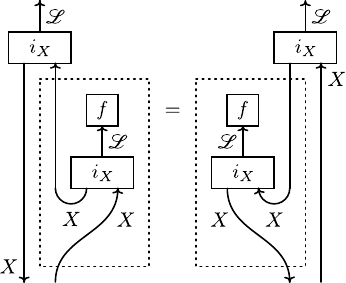}
 \]
 Then, it is enough to notice that the endomorphisms in the dashed boxes are dual to each other, so that the above equality is a consequence of dinaturality.
\end{proof}

Since both integrals and cointegrals for $\coend$ are two-sided in the unimodular case, we can drop the prefix left/right. We will need the following technical result.

\begin{lemma}\label{L:identities}
 Let $\intL$ and $\cointL$ be a non-zero integral and a non-zero cointegral of $\coend$ respectively.
 \begin{enumerate}
  \item $\cointL \circ \intL \neq 0$;
  \item $\antipL \circ \intL = \intL$;
  \item $(\antipL \otimes \id_\coend) \circ \copL \circ \intL = (\id_\coend \otimes \antipL) \circ \copL \circ \intL$.
 \end{enumerate}
\end{lemma}

\begin{proof}
 Part (\textit{i}) is proven in \cite[Thm.~4.2.5]{KL01} (since we assume unimodularity, the object of integrals is the tensor unit). Part (\textit{ii}) follows by composing Equation~(a) of \cite[Lem.~3.13]{DR12} with $\id_\coend \otimes \cointL$ from the left and with $\intL$ from the right (this is just a braided version of the argument showing \cite[Eq.~(2)]{R94}). To see part (\textit{iii}), first establish
 \begin{equation}\label{E:sec2pic2}
  \raisebox{-0.5\height}{\includegraphics{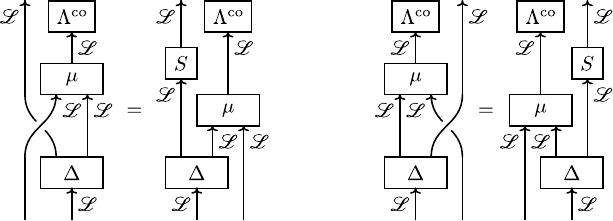}} 
 \end{equation}  
 which is a dual version of \cite[Lem.~3.13]{DR12}. Then, compose both equalities with $\intL \otimes \id_\coend$ from the right, compose the second one with $\antipL^{-1}$ from the left, and use the non-degeneracy of the pairing 
$\cointL \circ \multL : \coend \otimes \coend \to \one$, which follows from \cite[Cor.~4.2.13]{KL01}, together with with part (\textit{ii}) and with the identity
 \begin{equation}\label{E:coprod}
  \copL \circ \antipL = c_{\coend,\coend} \circ (\antipL \otimes \antipL) \circ \copL \qedhere
 \end{equation}
\end{proof}

For later use, we note that parts (\textit{ii}) and (\textit{iii}) of Lemma \ref{L:identities} imply
\begin{equation}\label{eq:brHopf-identity}
 (\antipL \otimes \antipL^{-1}) \circ \Delta \circ \intL 
 = \Delta \circ \intL 
 = \Delta \circ \antipL^{-1} \circ \intL.
\end{equation}

\subsection{Twist non-degeneracy}\label{S:non-degeneracy}

We introduce the \textit{$T$-transformation} $\modT : \coend \to \coend$ on the coend as the unique morphism satisfying, for all $V \in \cat$, the identity
\begin{equation}
 \modT \circ i_V = i_V \circ (\id_{V^*} \otimes \theta_V).
\end{equation}

\begin{definition}\label{D:twist_non-degenerate}
A unimodular finite ribbon category is {\em twist non-de\-gen\-er\-ate} if there exist non-zero constants $\Delta_\pm$ depending on the normalization of $\intL$ such that
\begin{equation}\label{eq:twist-nondeg}
 \counitL \circ \modT^{\pm1} \circ \intL = \Delta_{\pm} \id_{\one}.
\end{equation}
\end{definition}

\subsection{Modularity}\label{S:modularity}

A braided finite tensor category is called \textit{factorizable} if the Hopf pairing
$\pairL : \coend \otimes \coend \to \one$ defined by Equation \eqref{E:Hopf_pairing} is non-degenerate, meaning that $(\omega \otimes \id_{\coend^*}) \circ (\id_\coend \otimes \lcoev_\coend)$ gives an isomorphism between $\coend$ and $\coend^\ast$. Equivalently, a braided finite tensor category is factorizable if and only if all transparent objects are isomorphic to direct sums of the tensor unit \cite{S16}.

\begin{definition}\label{D:modular_cat}
 A \textit{modular category} is a finite ribbon category which is factorizable.
\end{definition}

We stress that a modular category in the above sense need not be semisimple.

\begin{proposition}\label{P:if_mod_then_non-deg_unimod}
A modular category is unimodular and twist non-de\-gen\-er\-ate.
\end{proposition}

\begin{proof}
Unimodularity is proved in \cite[Lem.~5.2.8]{KL01}. Twist non-de\-gen\-er\-a\-cy follows form the proofs of \cite[Lem.~4.2.11~\&~5.2.8]{KL01}, but we sketch the argument here for the convenience of the reader.

Define $p : \coend \otimes \coend \to \one$ as $p := \counitL \circ \modT^{-1} \circ \multL$. Substituting the
definition of $\multL$, $\modT$, and $\pairL$, one verifies that for all $V,W \in \cat$,
\[
 p \circ (i_V \otimes i_W) = \pairL \circ (\modT^{-1} \otimes \modT^{-1}) \circ (i_V \otimes i_W).
\]
This shows that $p$ is a non-de\-gen\-er\-ate pairing. In particular, $p \circ (\intL \otimes \id_\coend)$ is a non-zero morphism.  Substituting the definition of $p$ and using Equation \eqref{eq:right-integral-def} we find 
$p \circ (\intL \otimes \id_\coend) = \Delta_- \counitL$. This means $\Delta_- \neq 0$. The proof that $\Delta_+ \neq 0$ is analogous.
\end{proof}

For modular categories, we will later use (for the so-called \textit{Cutting Lemma}~\ref{L:cutting}) the following two properties of the cointegral $\cointL$. The first statement is proved in \cite[Thm.~5]{K96a}, see also \cite[Lem.~2.3]{FGR17a}.

\begin{lemma}\label{lem:coint-from-int}
 Let $\cat$ be unimodular and let $\intL$ and $\cointL$ be an integral and a cointegral of $\coend$ satisfying $\cointL \circ \intL = \id_{\one}$. The Hopf pairing $\pairL$ is non-degenerate if and only if there exists a non-zero coefficient $\zeta \in \Bbbk^*$ satisfying the equation
 \begin{equation}\label{E:mod_par}
  \pairL \circ ( \intL \otimes \id_\coend ) = \zeta \cointL.
 \end{equation}
\end{lemma}

We call $\zeta$ the \textit{modularity parameter} of $\intL$. Recall that both $\intL$ and $\cointL$ are unique up to scalar, and both $\zeta$ and $\cointL$ depend on the choice of the normalization of $\intL$.

The next statement was proven in \cite[Cor.~6.4]{GR17}.

\begin{lemma}\label{lem:lambda-iota}
	If $\cat$ is a modular category, then there exists a unique morphism $\eta_{\one} : \one \to P_{\one}$ satisfying, for every $V \in \Irr$, the equation
    \begin{equation}\label{eq:lambda-iota}
	 \cointL \circ i_{P_V} = \delta_{V,\one} \eta_{\one}^* \otimes \epsilon_{\one}.
	\end{equation}
\end{lemma}

\subsection{Drinfeld map}\label{S:Drinfeld}

Consider the dinatural transformation $\T$ whose component $\T_{X,Y} : X^* \otimes X \to Y \otimes Y^*$ is defined, for all objects $X,Y \in \cat$, as 
\[
 \raisebox{-0.5\height}{\includegraphics{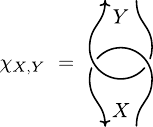}} 
\]
The universal properties of $\coend$ and of $\eend$ imply the existence of a unique morphism $\DD: \coend \to \eend$ satisfying $\T_{X,Y}  =  j_Y\circ \DD \circ i_X$. We call $\DD$ the {\em Drinfeld map} for the category $\calC$.
The following result was proven in \cite[Prop.~4.11]{FGR17a}.

\begin{proposition}\label{prop:C-DD}
A braided finite  tensor category is factorizable if and only if $\DD$ is invertible.
\end{proposition}

Let $k$ be the dinatural transformation whose component $k_X : \coend^* \to X \otimes X^*$ is given by $(i_{X^*})^*$ for every $X \in \calC$. Then $k$ determines a unique morphism ${\tilde{k} : \coend^* \to \eend}$ satisfying 
\begin{equation}\label{E:tilde_k}
 j_X \circ \tilde{k} = k_X
\end{equation}
for every $X \in \calC$. Similarly, let $\ell$ be the dinatural transformation whose component $\ell_X : X^* \otimes X \to \eend^*$ is given by $(j_{X^*})^*$ for every $X \in \calC$. Then $\ell$ determines a unique morphism $\tilde{\ell} : \coend \to \eend^*$ satisfying 
\begin{equation}\label{E:tilde_ell}
 \tilde{\ell} \circ i_X = \ell_X
\end{equation}
for every $X \in \calC$. These morphisms satisfy
\begin{equation}
 \lev_{\eend} \circ (\tilde{\ell} \otimes \tilde{k}) = \rev_{\coend}, \quad
 (\tilde{k} \otimes \tilde{\ell}) \circ \rcoev_{\coend} = \lcoev_{\eend}.
\end{equation}

\begin{lemma}\label{L:inverse_Drinfeld}
 If $\intL$ and $\cointL$ are an integral and a cointegral of $\coend$ satisfying ${\cointL \circ \intL = \id_{\one}}$, then the morphism $\tilde{D} : \eend \to \coend$ defined by
 \[
  \raisebox{-0.5\height}{\includegraphics{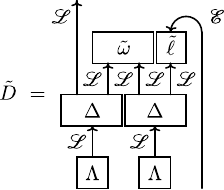}} 
 \]
 satisfies $\tilde{D} \circ \DD = \zeta \id_\coend$ for the modularity parameter $\zeta \in \Bbbk^*$.
\end{lemma}

\begin{proof}
 First of all, we claim 
 \begin{equation}\label{E:inverse_Drinfeld_1}
  \lev_\eend \circ (\tilde{\ell} \otimes \DD) = \pairL.
 \end{equation}
 To see this, remark that, for all $X,Y \in \calC$, we have
 \begin{equation}
  \raisebox{-0.5\height}{\includegraphics{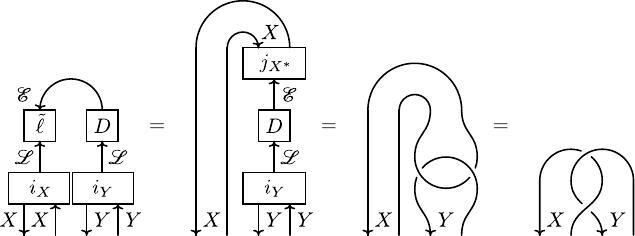}} 
 \end{equation}
 Next, recall that by definition $\tilde{\pairL} = \pairL \circ (\id_\coend \otimes \antipL)$, see above Equation \eqref{E:mirrored_pairing}. Then, the proof is given by
 \begin{align*}
  \includegraphics{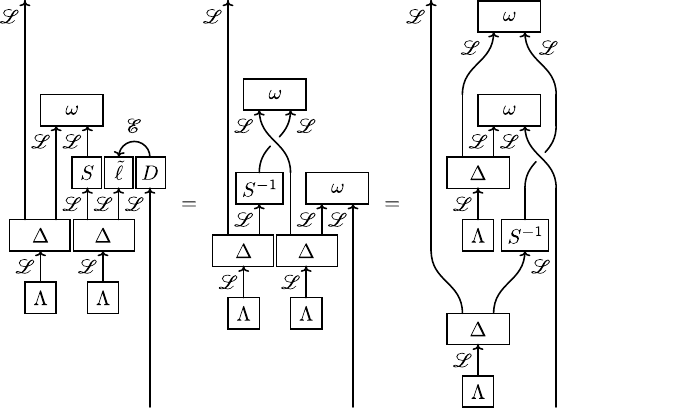} \\[0.5\baselineskip]
  \includegraphics{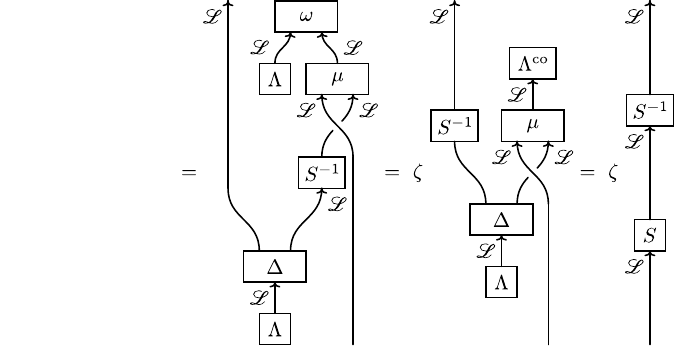}
 \end{align*}
 where the first equality follows from Equations \eqref{E:Hopf_3} and \eqref{E:inverse_Drinfeld_1}, the second one from naturality of the braiding, the third one from Equation \eqref{E:Hopf_1}, the fourth one from Lemma \ref{lem:coint-from-int} and from part (\textit{iii}) of Lemma \ref{L:identities}, and the last one from the left-hand side of Equation \eqref{E:sec2pic2}, by composing both morphisms with $\id_\coend \otimes \intL$ from the right, and by using naturality of the braiding.
\end{proof}

\subsection{Semisimple case}\label{S:ssi}

We finish this first part by discussing what all the above notions and conditions amount to in case $\calC$ is semisimple, so let us assume this for the remainder of the section. Let us start with ideals and traces. First of all, remark that the unit object $\one \in \calC$ is projective. In fact, conversely, if $\one \in \calC$ is projective, then $\calC$ is semisimple, see \cite[Cor.~4.2.13]{EGNO15}. Now, since ideals are absorbent under tensor products, $\Proj(\calC) = \calC$, and so \cite[Lem.~17]{GPV11} implies $\calC$ is the only non-zero ideal of $\calC$. Therefore Proposition \ref{P:uniqueness_of_trace} implies the categorical trace $\tr_\calC$ defined by
\[
 \tr_\calC(f) := \rev_V \circ \left( f \otimes \id_{V^*} \right) \circ \lcoev_V
\]
for every $f \in \End_\calC(V)$ is the unique non-zero trace on $\calC$, up to scalar.

For what concerns coends, \cite[Lem.~2]{K96a} gives
\[
 \coend = \bigoplus_{V \in \Irr} V^* \otimes V,
\]
and if we denote with $\iota_V \in \calC(V^* \otimes V,\coend)$ the canonical injection morphism for every $V \in \Irr$, then for every object $X \in \calC$ the component $i_X \in \calC(X^* \otimes X,\coend)$ of the dinatural transformation $i$ is given by
\[
 i_X = \sum_{i=1}^m \iota_{V_i} \circ \left( g_i^* \otimes f_i \right),
\]
where the object $V_i \in \Irr$ and the morphisms $f_i \in \calC(X,V_i)$ and $g_i \in \calC(V_i,X)$ satisfy
\[
 \id_X = \sum_{i=1}^m g_i \circ f_i.
\]

Next, \cite[Sec.~2.5]{K96a} implies an integral $\intL \in \calC(\one,\coend)$ and a cointegral $\cointL \in \calC(\coend,\one)$ satisfying $\cointL \circ \intL = \id_{\one}$ are given by
\[
 \intL = \sum_{V \in \Irr} \dim_\calC(V) \iota_V \circ \rcoev_V, \qquad
 \cointL = \pi_{\one},
\]
where $\pi_V \in \calC(\coend,V^* \otimes V)$ denotes the canonical projection morphism for every $V \in \Irr$, and $\dim_\calC(V) := \tr_\calC(\id_V)$ for every $V \in \calC$.  Note that this choice of normalization of $\intL$ is not the same one as in \cite{K96a}. It is rather the one we will need in the following in order to obtain the standard normalization of \textit{Kirby colors}, in the language of \cite{B03}, when applying our construction to the semisimple case. With this choice, twist non-degeneracy of $\calC$ translates to the condition
\[
 \Delta_{\pm} = \sum_{V \in \Irr} \langle \theta_V \rangle^{\pm1} \dim_\cat(V)^2 \neq 0,
\]
where for every $V \in \Irr$ and every $f \in \End_\calC(V)$ the scalar $\langle f \rangle \in \Bbbk$ is defined by $f = \langle f \rangle \id_V$, and $\intL$ and $\cointL$ satisfy Equation \eqref{E:mod_par} with
\[
 \zeta = \sum_{V \in \Irr} \dim_\cat(V)^2,
\]
see \cite[Rem.~3.10.3]{GR16}.

\section{Closed 3-manifold invariants}\label{S:3-manifold_invariants}

In this section we renormalize Lyubashenko invariants of closed 3-manifolds \cite{L94} through the use of modified traces. Our construction applies to all twist non-degenerate finite unimodular ribbon categories, although some of these conditions can be relaxed for most of the intermediate steps. Indeed, as a preparation, we define a functor $F_\intL : \calR_\intL \to \cat$, in Proposition \ref{P:LRT_functor}, and a link invariant $F'_{\intL,\rmt}$, in Theorem \ref{T:admissible_closed_graph_invariant}, whose construction does not require twist non-degeneracy. This additional hypothesis is first needed for the definition of the 3-manifold invariant $\rmL'_{\calC,\Tideal}$ in Theorem \ref{T:admissible_closed_3-manifold_invariant}. Therefore, unless stated otherwise, in this section we adopt the following convention:

\begin{quote}
 $\calC$ is a finite unimodular ribbon category.
\end{quote}

\subsection{Lyubashenko-Reshetikhin-Turaev functor}\label{SS:LRT_functor}

We start with the definition of a monoidal functor from a category of partially $\calC$-labeled ribbon graphs, featuring red and blue edges, to $\calC$. Red and blue edges play different roles in the construction: Red edges are related to the Lyubashenko invariant, they are unlabeled, and they are to be evaluated using the right integral $\intL$ on the coend $\coend$; Blue edges are more standard, they can be labeled with any object of $\calC$, and they are to be evaluated using the Reshetikhin-Turaev functor $F_\calC$. See Turaev's book \cite{T94} for a reference about ribbon graphs and Reshetikhin-Turaev functors.

By a \textit{closed manifold} we mean a compact manifold without boundary. Every manifold we will consider will be oriented, every diffeomorphism of manifolds will be positive, and all links and tangles will be both oriented and framed. If $M$ is a manifold, then $\overline{M}$, or sometimes also $(-1)M$, will denote the manifold obtained from $M$ by reversing its orientation. The interval $[0,1]$ will be denoted $I$.

An \textit{$(n,n')$-tangle} is a tangle with $n$ incoming boundary vertices and $n'$ outgoing ones. An \textit{$n$-bottom tangle} is a $(2n,0)$-tangle whose $2k-1$th incoming boundary vertex is connected to its $2k$th incoming boundary vertex by an edge directed from right to left for every $1 \leq k \leq n$.

An \textit{$n$-bottom graph} is a partially $\calC$-labeled ribbon graph with edges divided into two groups, \textit{red} and \textit{blue}, and with coupons coming in two flavors, \textit{bichrome} and \textit{blue}, satisfying the following conditions:
\begin{enumerate}
 \item The $2n$ leftmost incoming boundary vertices are red, while all the other incoming and outgoing boundary vertices are blue;
 \item Red edges are unlabeled, while blue edges are labeled, as is usual in Turaev's approach, with objects of $\calC$;
 \item Bichrome coupons are unlabeled, and they are quite rigid, meaning the only possible configurations are the ones represented in Figure \ref{F:smoothing}, with fixed number, orientations, and labels of blue edges, while standard coupons are labeled as usual with morphisms of $\calC$, and they are more flexible, with the only condition that they exclusively meet with blue edges;
 \item The operation of \textit{smoothing}, which consists in throwing away all blue edges and coupons and by replacing every bichrome coupon with red strands connecting red edges as shown in Figure \ref{F:smoothing}, produces a red $n$-bottom tangle.
\end{enumerate}

\begin{figure}[ht]
 \centering
 \includegraphics{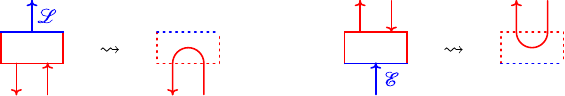}
 \caption{Bichrome coupons and their smoothing.}
 \label{F:smoothing}
\end{figure}

A $0$-bottom graph is simply called a \textit{bichrome graph}. See Figure \ref{F:example} for an example of a $1$-bottom graph together with its smoothing.

\begin{figure}[ht]
 \centering
 \includegraphics{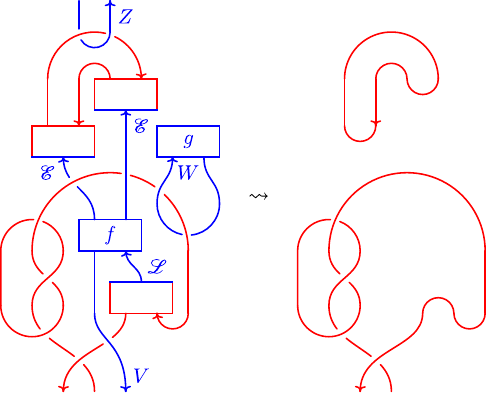}
 \caption{A $1$-bottom graph and its smoothing.}
 \label{F:example}
\end{figure}

Next, let us define the \textit{ribbon category $\calR_\intL$ of bichrome graphs\footnote{The notation $\calR_\intL$ was first used in \cite{DGP17}, and it refers to the fact that bichrome graphs feature red edges, which have to be evaluated using the integral $\intL$, as explained in the following. Remark however that $\calR_\intL$ is actually independent of the choice of a normalization for $\intL$, which is unique only up to a scalar coefficient. Remark also that our definition of $\calR_\intL$ is slightly simpler, as bichrome coupons are more rigid.}}. An object $(\underline{\epsilon},\underline{V})$ of $\calR_\intL$ is a finite sequence $((\epsilon_1,V_1), \ldots, (\epsilon_m,V_m))$ where $\epsilon_k \in \{ +, - \}$ is a sign and $V_k$ is an object of $\calC$ for every integer $1 \leq k \leq m$. Every object $(\underline{\epsilon},\underline{V})$ of $\calR_\intL$ determines a standard set of $\calC$-labeled blue framed oriented points inside $\R^2$, which we still denote with $(\underline{\epsilon},\underline{V})$. Then, a morphism $T : (\underline{\epsilon},\underline{V}) \to (\underline{\epsilon'},\underline{V'})$ of $\calR_\intL$ is an isotopy class of bichrome graphs in $\R^2 \times I$ from $(\underline{\epsilon},\underline{V}) \times \{ 0 \}$ to $(\underline{\epsilon'},\underline{V'}) \times \{ 1 \}$, with matching framings, orientations, and labels.

A morphism of $\calR_\intL$ is \textit{blue} if it features no red edge. The \textit{category $\calR_\calC$ of blue graphs} is the subcategory of $\calR_\intL$ having the same objects, but featuring only blue morphisms. The Reshetikhin-Turaev functor $F_\calC$ is naturally defined on $\calR_\calC$.

We denote with $(n)\calR_\intL((\underline{\epsilon},\underline{V}),(\underline{\epsilon'},\underline{V'}))$ the set of isotopy classes of $n$-bottom graphs from $(n)(\underline{\epsilon},\underline{V}) \times \{ 0 \}$ to $(\underline{\epsilon'},\underline{V'}) \times \{ 1 \}$, where $(n)(\underline{\epsilon},\underline{V})$ is obtained from $(\underline{\epsilon},\underline{V})$ by adding $2n$ unlabeled red incoming boundary vertices to the left, with negative orientation in odd positions and positive orientation in even positions. Note that $(n)(\underline{\epsilon},\underline{V})$ is not an object of $\calR_\intL$, and that $(n)\calR_\intL((\underline{\epsilon},\underline{V}),(\underline{\epsilon'},\underline{V'}))$ is not a morphism space in $\calR_\intL$.

\FloatBarrier

If $(\underline{\epsilon},\underline{V})$ and $(\underline{\epsilon'},\underline{V'})$ are objects of $\calR_\intL$ let us consider the \textit{plat closure map}
\[
 \begin{array}{rccc}
  \pc : & (n)\calR_\intL((\underline{\epsilon},\underline{V}),(\underline{\epsilon'},\underline{V'})) & \to & \calR_\intL((\underline{\epsilon},\underline{V}),(\underline{\epsilon'},\underline{V'})) \\ & & & \\
  & \raisebox{-0.5\height}{\includegraphics{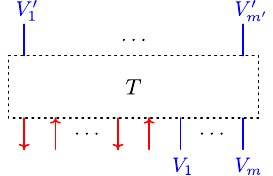}} & \mapsto 
  & \raisebox{-0.5\height}{\includegraphics{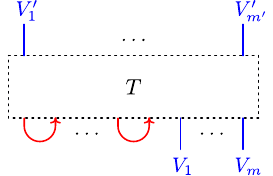}}
 \end{array}
\]

When $T \in \calR_\intL((\underline{\epsilon},\underline{V}),(\underline{\epsilon'},\underline{V'}))$ is a bichrome graph then we say an $n$-bottom graph $\tilde{T} \in (n)\calR_\intL((\underline{\epsilon},\underline{V}),(\underline{\epsilon'},\underline{V'}))$ whose smoothing has no closed components is a \textit{complete $n$-bottom graph presentation of $T$} if
\[
 \pc(\tilde{T}) = T.
\]
We also say a set $C$ of red edges of a morphism $T$ of $\calR_\intL$ is a \textit{chain} if all of its elements are contained in one and the same component of the smoothing of $T$. We use the term \textit{cycle} to denote a maximal chain in $T$ with respect to inclusion. For instance, the plat closure of the left-hand side of Figure~\ref{F:example} features two distinct cycles, one composed of the red edge intersecting the bichrome tangle with blue edge labeled by $\coend$, and another composed of the remaining two red edges.

We can now explain how to extend the Reshetikhin-Turaev functor $F_\calC : \calR_\cat \to \cat$ to a functor $F_\intL : \calR_\intL \to \cat$. If $(\underline{\epsilon},\underline{V})$ is an object of $\calR_\intL$, then we set 
\[
 F_\intL(\underline{\epsilon},\underline{V}) := F_\calC(\underline{\epsilon},\underline{V}).
\]
If $T : (\underline{\epsilon},\underline{V}) \to (\underline{\epsilon'},\underline{V'})$ is a morphism of $\calR_\intL$, then every complete $n$-bottom graph presentation $\tilde{T}$ of $T$ induces an $n$-dinatural transformation 
\[
 \eta_{\tilde{T}} : H_{F_\calC(\underline{\epsilon},\underline{V})} \din F_\calC(\underline{\epsilon'},\underline{V'}),
\]
where $F_\calC(\underline{\epsilon'},\underline{V'}) : (\calC^\op \times \calC)^{\times n} \to \calC$ denotes the constant functor sending every object $(X_1,Y_1,\ldots,X_n,Y_n) \in (\calC^\op \times \calC)^{\times n}$ to the object $F_\calC(\underline{\epsilon'},\underline{V'}) \in \calC$, where
\[
 H_{F_\calC(\underline{\epsilon},\underline{V})} : (\calC^\op \times \calC)^{\times n} \to \calC
\]
sends every object $(X_1,Y_1,\ldots,X_n,Y_n) \in (\calC^\op \times \calC)^{\times n}$ to the object 
\[
 X_1^* \otimes Y_1 \otimes \ldots \otimes X_n^* \otimes Y_n \otimes F_\calC(\underline{\epsilon},\underline{V}) \in \calC,
\]
and where $n$-dinaturality means $\eta_{\tilde{T}}$ defines a dinatural transformation 
\[
 \eta_{\tilde{T}}^\sigma : H_{F_\calC(\underline{\epsilon},\underline{V})} \circ \sigma \din F_\calC(\underline{\epsilon'},\underline{V'}) \circ \sigma
\]
for the permutation functor 
\[
 \sigma : (\calC^{\times n})^\op \times \calC^{\times n} \to (\calC^\op \times \calC)^{\times n}
\]
sending every object $((X_1,\ldots,X_n),(Y_1,\ldots,Y_n)) \in (\calC^{\times n})^\op \times \calC^{\times n}$ to the object $(X_1,Y_1,\ldots,X_n,Y_n) \in (\calC^\op \times \calC)^{\times n}$. The $n$-dinatural transformation $\eta_{\tilde{T}}$ associates with every object $(X_1, \ldots, X_n) \in \calC^{\times n}$ the morphism
\[
 F_\calC(\tilde{T}_{(X_1,\ldots,X_n)}) \in \calC(X_1^* \otimes X_1 \otimes \ldots \otimes X_n^* \otimes X_n \otimes F_\calC(\underline{\epsilon},\underline{V}),F_\calC(\underline{\epsilon'},\underline{V'})),
\]
where $\tilde{T}_{(X_1,\ldots,X_n)}$ is the ribbon graph obtained from the $n$-bottom graph $\tilde{T}$ by labeling its $k$th cycle with $X_k$, by labeling every bichrome coupon intersecting it with either $i_{X_k}$ or $j_{X_k}$, the structure morphisms of $\coend$ and $\eend$ defined in Section \ref{S:coend}, for every integer $1 \leq k \leq n$, and by forgetting the distinction between red and blue. The universal property defining $\coend$ implies the object $\coend^{\otimes n} \otimes F_\calC(\underline{\epsilon},\underline{V})$ equipped with the dinatural transformation $i^{\otimes n} \otimes \id_{F_\calC(\underline{\epsilon},\underline{V})}$ is the coend for the functor ${H_{F_\calC(\underline{\epsilon},\underline{V})} \circ \sigma}$. This determines a unique morphism ${f_\calC(\eta_{\tilde{T}}) \in \calC(\coend^{\otimes n} \otimes F_\calC(\underline{\epsilon},\underline{V}),F_\calC(\underline{\epsilon'},\underline{V}'))}$ satisfying
\begin{equation}\label{E:f_calC}
 f_\calC(\eta_{\tilde{T}}) \circ (i_{X_1} \otimes \ldots \otimes i_{X_n} \otimes \id_{F_\calC(\underline{\epsilon},\underline{V})}) = F_\calC(\tilde{T}_{(X_1,\ldots,X_n)}).
\end{equation}
Then we define $F_\intL (T) : F_\calC(\underline{\epsilon},\underline{V}) \to F_\calC(\underline{\epsilon'},\underline{V'})$ as
\begin{equation}\label{E:LRT_functor}
 F_\intL (T) := 
 \pic{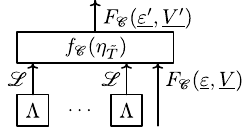}
\end{equation}

\begin{proposition}\label{P:LRT_functor}
 $F_\intL : \calR_\intL \to \calC$ is a well-defined monoidal functor.
\end{proposition}

Since the proof of Proposition \ref{P:LRT_functor} is rather long and technical, we postpone it to Appendix \ref{S:proof_LRT_functor}. We call $F_\intL : \calR_\intL \to \calC$ the \textit{Lyubashenko-Reshetikhin-Turaev functor} associated with the integral $\intL$. Recall that $\intL$ is unique up to scalar, and $F_\intL$ depends on the choice of its normalization.

\subsection{Renormalized Lyubashenko invariant of closed 3-manifolds}\label{SS:3-manifold_invariants}

For the next step of the construction, we need to fix an ideal $\calI$ in $\calC$ together with a trace $\rmt$ on $\calI$, so let us suppose such a $\rmt$ exists, and let us fix a choice, in case it is not unique. This key ingredient allows for the definition of a renormalized invariant of closed bichrome graphs which satisfy a certain admissibility condition. Indeed, in order to use the trace $\rmt$, we need a blue edge whose label is an object of $\calI$. With this in place, we can define a renormalized Lyubashenko invariant of closed 3-manifolds equipped with embedded admissible closed bichrome graphs. Theorems \ref{T:admissible_closed_graph_invariant} and \ref{T:admissible_closed_3-manifold_invariant} prove the existence of such invariants.

We say a bichrome graph is \textit{closed} if it features no boundary vertex, and we say it is \textit{admissible} if it features a blue edge whose label is an object of $\calI$.

\begin{remark}\label{R:admissible_ssi}
 When $\calC$ is semisimple, a bichrome graph is admissible if and only if it is non-empty, because every ideal in $\calC$ coincides with $\calC$ itself. Therefore, in that case, we assume the admissibility condition to be void, since for our purposes an empty graph can always be replaced by an unknot with label $\one$.
\end{remark}

If $T$ is a closed admissible bichrome graph and $V$ is an object of $\calI$, we say an endomorphism $T_V$ of $(+,V)$ in $\calR_\intL$ is a \textit{cutting presentation of $T$} if 
\[
 \tr_{\calR_\intL}(T_V) = T,
\]
where $\calR_\intL$ inherits its ribbon structure directly from $\calR_\calC$.

\begin{theorem}\label{T:admissible_closed_graph_invariant}
 If $T$ is an admissible closed bichrome graph and $T_V$ is a cutting presentation of $T$ then
 \[
  F'_{\intL,\rmt}(T) := \rmt_V(F_\intL(T_V))
 \]
 is an invariant of the isotopy class of $T$.
\end{theorem}

\begin{proof}
 The proof is similar to the one provided by \cite{GPT07}. Indeed, if $T_V$ and $T_W$ are two different cutting presentations of $T$, we can find an endomorphism $T_{V,W}$ of $((+,V),(+,W))$ such that 
 \[
  \rptr(T_{V,W}) = T_V, \qquad \lptr(T_{V,W}) = T_W.
 \]
 Roughly speaking, if $T_V$ and $T_W$ are obtained from $T$ by cutting at two different edges, then $T_{V,W}$ is obtained by cutting at both edges. Then the properties of the modified trace imply 
 \[
  \rmt_V (F_\intL(T_V)) = \rmt_{V \otimes W} (F_\intL(T_{V,W})) = \rmt_W (F_\intL(T_W)). \qedhere
 \]
\end{proof}

We call $F'_{\intL,\rmt}$ the \textit{renormalized invariant of admissible closed bichrome graphs}. Recall that $\intL$ is unique up to scalar, and $F'_{\intL,\rmt}$ depends both on the choice of the normalization of $\intL$ and on the choice of $\rmt$. Let us illustrate with an example the difference between $F_\intL$ and $F'_{\intL,\rmt}$ in the case $\calI = \Proj(\calC)$: if $T$ denotes the admissible closed blue graph
\begin{equation}\label{E:renormalized_invariant_example}
 \raisebox{-0.5\height}{\includegraphics{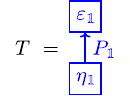}}
\end{equation}
then we have
\[
 F_\intL(T) = \epsilon_{\one} \circ \eta_{\one} = 0,
\]
while
\[
 F'_{\intL,\rmt}(T) = \rmt_{P_{\one}}(\eta_{\one} \circ \epsilon_{\one}) \neq 0,
\]
because $\calC(P_{\one},\one)$ is generated by $\epsilon_{\one}$, $\calC(\one,P_{\one})$ is generated by $\eta_{\one}$, and $\rmt$ is non-degenerate, as stated in Proposition \ref{P:uniqueness_of_trace}.

\begin{proposition}\label{P:FF'}
 If $T$ and $T'$ are closed bichrome graphs and $T'$ is admissible, then 
 \[
  F'_{\intL,\rmt}(T \otimes T') = F_\intL(T) F'_{\intL,\rmt}(T').
 \]
\end{proposition}

\begin{proof}
 If $T'_V$ is a cutting presentation of $T'$, then $T \otimes T'_V$ is a cutting presentation of $T \otimes T'$, and the proposition follows from the fact that 
 \[
  F_\intL(T \otimes T'_V) = F_\intL(T)F_\intL(T'_V). \qedhere
 \]
\end{proof}

Next, let us state a couple of key properties of the functor $F_\intL$ which will be crucial for the construction.

\begin{proposition}\label{P:orient} 
 If $T$ is a morphism of $\calR_\intL$, if $K$ is a red cycle of $T$ which does not intersect bichrome coupons, and if $\tilde{T}$ denotes the morphism of $\calR_\intL$ obtained by reversing the orientation of $K$, then
 \[
  F_\intL(T) = F_\intL(\tilde{T}).
 \]
 Similarly, if $T$ is also closed and admissible, then $F'_{\intL,\rmt}(T) = F'_{\intL,\rmt}(\tilde{T})$.
\end{proposition}

\begin{remark}\label{R:orient}
 In Section \ref{S:pairing}, we need a slightly more general version of Proposition \ref{P:orient}. Indeed, both $F_\intL$ and $F'_{\intL,\rmt}$ are actually invariant under orientation reversal of arbitrary red cycles, regardless of their intersection with bichrome coupons. Of course, we need to say what it means to reverse the orientation of a red cycle which is not disjoint from bichrome coupons. This is explained in Figure \ref{F:orientation_reversal_coupons} using the morphisms $\tilde{k} : \coend^* \to \eend$ and $\tilde{\ell} : \coend \to \eend^*$ defined by Equations \eqref{E:tilde_k} and \eqref{E:tilde_ell}. The proof of the general statement is identical to the proof of Proposition \ref{P:orient}.
\end{remark}

\begin{figure}[h]
 \centering
 \includegraphics{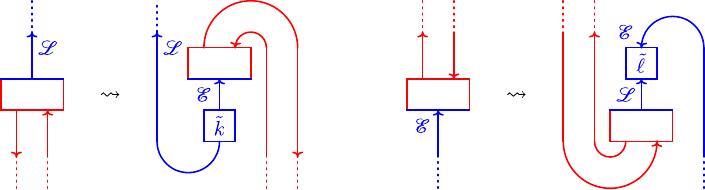}
 \caption{Orientation reversal of bichrome coupons.}
 \label{F:orientation_reversal_coupons}
\end{figure}

\begin{proposition}\label{P:slide}
 If $T$ is a morphism of $\calR_\intL$, if $K$ is a red cycle of $T$ which does not intersect bichrome coupons, and if $\tilde{T}$ denotes the morphism of $\calR_\intL$ obtained by sliding\footnote{See (the plat closure of) Figure~\ref{F:edge_slide} for the operation of edge slide.} a blue or a red edge of $T$ over $K$, then
 \[
  F_\intL(T) = F_\intL(\tilde{T}).
 \]
 Similarly, if $T$ is also closed and admissible, then $F'_{\intL,\rmt}(T) = F'_{\intL,\rmt}(\tilde{T})$.
\end{proposition}

Since the proof of Propositions \ref{P:orient} and \ref{P:slide} is very similar in spirit to the proof of Proposition \ref{P:LRT_functor}, we will postpone it to Appendix \ref{S:proof_orient_slide}. Now, recall the stabilization coefficients $\Delta_+$ and $\Delta_-$ of Definition \ref{D:twist_non-degenerate}. Assuming $\calC$ is twist non-degenerate, meaning $\Delta_+ \Delta_- \neq 0$, we can fix coefficients $\calD, \delta \in \Bbbk^*$ satisfying
\[
 \calD^2 = \Delta_+ \Delta_-, \qquad \delta = \frac{\calD}{\Delta_-} = \frac{\Delta_+}{\calD}.
\]

\begin{theorem}\label{T:admissible_closed_3-manifold_invariant}    
 If $\calC$ is twist non-degenerate, if $M$ is a closed connected 3-manifold, and if $T$ is an admissible closed bichrome graph inside $M$, then
 \[
  \rmL'_{\calC,\Tideal}(M,T) := \calD^{-1 - \ell} \delta^{- \sigma(L)} F'_{\intL,\rmt}(L \cup T)
 \]
 is a topological invariant of the pair $(M,T)$, with $L$ being a surgery presentation of $M$ given by a red $\ell$-component link of signature $\sigma(L)$ inside $S^3$.
\end{theorem}

\begin{proof}
 The proof follows the argument of Reshetikhin and Turaev, showing that the quantity $\calD^{-1 - \ell} \delta^{- \sigma(L)} F'_{\intL,\rmt}(L \cup T)$ remains unchanged under orientation reversal of components of $L$ and under Kirby moves \cite{RT91}. First of all, Proposition \ref{P:orient} implies $F'_{\intL,\rmt}(L \cup T)$ is independent of the choice of the orientation of the surgery link $L$. Then, thanks to Proposition \ref{P:slide}, $F'_{\intL,\rmt}(L \cup T)$ is also invariant under handle slides, known as Kirby II moves\footnote{A Kirby II move corresponds to a special case of the edge slide operation, when the edge being slid is a red surgery component.}. Finally, the invariance of $\rmL'_{\calC,\Tideal}(M,T)$ under stabilizations, known as Kirby I moves\footnote{A Kirby I move corresponds to the creation or removal of a disjoint unknotted red surgery component of framing $\pm1$.}, follows from the choice of the normalization factor $\calD^{-1 - \ell} \delta^{- \sigma(L)}$, which is made possible by the twist non-degeneracy of $\calC$.
\end{proof}

We call $\rmL'_{\calC,\Tideal}$ the \textit{renormalized Lyubashenko invariant of admissible closed 3-manifolds}. Recall that $\intL$ is unique up to scalar, and $\rmL'_{\calC,\Tideal}$ depends both on the choice of the normalization of $\intL$ and on the choice of $\rmt$.

\begin{remark}
 The notation $F'_{\intL,\rmt}(L \cup T)$ is slightly abusive, because $T$ is actually contained in $M$. What we actually mean is that we have a diffeomorphism between $S^3(L)$ and $M$, and that $T$ can be isotoped to be inside the image of the exterior of $L$ in $S^3$ under this diffeomorphism. We can therefore pull back $T$ to an admissible closed bichrome graph inside $S^3$ which does not intersect $L$, and which we still denote with $T$.
\end{remark}

\begin{remark}\label{R:3-man_inv_ssi}
 When the category $\calC$ is semisimple then, thanks to Remark \ref{R:admissible_ssi}, $\rmL'_{\calC,\Tideal}$ recovers the standard Reshetikhin-Turaev invariant associated with $\calC$, see Section \ref{S:ssi}.
\end{remark}

The next result establishes a relation between our invariant $\rmL'_{\calC,\Tideal}$ and the standard Lyubashenko invariant $\rmL_\calC$, whose defining formula is obtained from the one of $\rmL'_{\calC,\Tideal}$ by replacing $F'_{\intL,\rmt}$ with $F_\intL$.

\begin{proposition}\label{P:connected_sum}
 If $\calC$ is twist non-degenerate, if $M$ and $M'$ are closed connected 3-manifolds, and if $T'$ is an admissible bichrome graph inside $M'$, then 
 \[
  \rmL'_{\calC,\Tideal}(M \# M',T') = \rmL_\calC(M) \rmL'_{\calC,\Tideal}(M',T').
 \]
\end{proposition}

\begin{proof}
 It is enough to apply Proposition \ref{P:FF'} to a surgery presentation for $M \# M'$ given by the disjoint union of two surgery presentations for $M$ and $M'$.
\end{proof}

\FloatBarrier

\section{2+1-TQFTs}\label{S:TQFTs}

In this section we extend the topological invariants constructed in Section \ref{S:3-manifold_invariants} to TQFTs, and we provide an explicit characterization of the resulting state spaces. In order to do this, we will focus on the case of modular categories and ideals of projective objects. Let us start by gathering a short list of the ingredients we introduced up to now that will be used in the following:
\begin{enumerate}
 \item $\calC$ is a modular category, or in other words a finite factorizable ribbon category. In particular, $\calC$ is unimodular and twist non-degenerate, and the Drinfeld map $\DD : \coend \to \eend$ between coend $\coend$ and end $\eend$ is an isomorphism, see Propositions \ref{P:if_mod_then_non-deg_unimod} and \ref{prop:C-DD}.
 \item $\intL$ is a fixed integral of $\coend$, which uniquely determines both a cointegral $\cointL$ of $\coend$ satisfying $\cointL \circ \intL = \id_{\one}$ and a modularity parameter $\zeta \in \Bbbk^*$, see Lemma \ref{lem:coint-from-int}.
 \item $F_\intL$ is the Lyubashenko-Reshetikhin-Turaev functor associated with the ribbon category $\calC$ and with the integral $\intL$ of the coend $\coend$, see Proposition~\ref{P:LRT_functor}.
 \item $\calI = \Proj(\calC)$ is the ideal of projective objects of $\calC$. It admits a trace $\rmt$ which is non-degenerate and unique up to scalar, see Proposition \ref{P:uniqueness_of_trace}. We fix the normalization $\rmt_{P_{\one}}(\eta_{\one} \circ \epsilon_{\one}) = 1$, see Lemma \ref{lem:lambda-iota}.
 \item $F'_\intL = F'_{\intL,\rmt}$ is the renormalized invariant of admissible closed bichrome graphs associated with the ribbon category $\calC$ and with the trace $\rmt$ on $\Proj(\calC)$, see Theorem \ref{T:admissible_closed_graph_invariant}.
 \item $\rmL'_\calC = \rmL'_{\calC,\Tideal}$ is the renormalized Lyubashenko invariant of admissible closed 3-man\-i\-folds associated with $F'_\intL$, see Theorem \ref{T:admissible_closed_3-manifold_invariant}.
\end{enumerate}

\subsection{Algebraic state spaces}\label{S:pairing}

We start with the definition of a family of vector spaces which will be later identified with state spaces coming from the functorial extension of the invariant $\rmL'_\calC$. For every integer $g \geq 0$ and for every object $V \in \calC$ we consider vector spaces
\[
 \calX_{g,V} := \calC(P_{\one},\eend^{\otimes g} \otimes V), \quad
 \calX'_{g,V} := \calC(\coend^{\otimes g} \otimes V,\one).
\]
We define now a bilinear pairing $\brk{\cdot,\cdot}_{g,V} : \calX'_{g,V} \times \calX_{g,V} \to \Bbbk$ as follows: for every $f' \in \calX'_{g,V}$ and every $f \in \calX_{g,V}$ we set
\[
 \brk{f',f}_{g,V} :=
 \rmt_{P_{\one}} \left( \pic{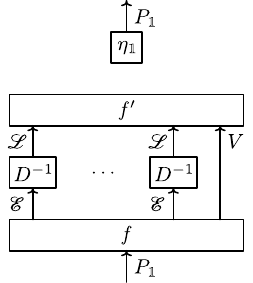} \right),
\]
where $\eta_{\one} : \one \to P_{\one}$ is the morphism fixed in Lemma~\ref{lem:lambda-iota}. Let now $\rmX_{g,V}$ be the quotient of the vector space $\calX_{g,V}$ with respect to the right radical of the bilinear form $\langle \cdot , \cdot \rangle_{g,V}$, and similarly let $\rmX'_{g,V}$ be the quotient of the vector space $\calX'_{g,V}$ with respect to the left radical of the bilinear form $\langle \cdot , \cdot \rangle_{g,V}$. Then the pairing $\langle \cdot , \cdot \rangle_{g,V}$ induces a non-degenerate pairing 
\[
 \langle \cdot , \cdot \rangle_{g,V} : \rmX'_{g,V} \otimes \rmX_{g,V} \rightarrow \Bbbk. 
\]

\begin{lemma}\label{L:left-non-degeneracy}
 For every integer $g \geq 0$ and every object $V \in \calC$ we have $\rmX'_{g,V} = \calX'_{g,V}$. 
\end{lemma}

\begin{proof}
 Let us consider a non-trivial vector $f' \in \calX'_{g,V}$, and let us show there exists a vector $f \in \calX_{g,V}$ satisfying $\brk{f',f}_{g,V} \neq 0$. First of all, the composition $\eta_{\one} \circ f'$ is a non-zero morphism of $\calC(\coend^{\otimes g} \otimes V,P_{\one})$. Then, since the trace $\rmt$ is non-degenerate, there exists a morphism $\tilde{f} \in \calC(P_{\one},\coend^{\otimes g} \otimes V)$ satisfying ${\rmt_{P_{\one}}(\eta_{\one} \circ f' \circ \tilde{f}) \neq 0}$. Therefore, we can simply set $f := (\DD^{\otimes g} \otimes \id_V) \circ \tilde{f}$.
\end{proof}

\begin{remark}
 The equality $\rmX_{g,V} = \calX_{g,V}$ holds for every integer $g \geq 0$ and every object $V \in \calC$ if and only if $\calC$ is semisimple.
\end{remark}

\begin{lemma}\label{L:comparison_of_pairings}
 For every $f' \in \calX'_{g,V}$ and every $f \in \calX_{g,V}$ the admissible bichrome graph $T_{f',f}$ represented in Figure \ref{F:comparison_of_pairings} satisfies
 \[
  F'_\intL \left( T_{f',f} \right) = \zeta^g \langle f' , f \rangle_{g,V}.
 \]
\end{lemma}

\begin{figure}[t]
 \centering
 \includegraphics{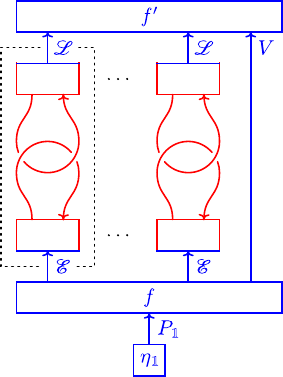}
 \caption{The admissible bichrome graph $T_{f',f}$.}
 \label{F:comparison_of_pairings}
\end{figure}

\begin{proof}
 Let $\tilde{T}$ denote the bichrome graph contained inside the dashed box in Figure~\ref{F:comparison_of_pairings}. Up to reversing the orientation of its bottom red cycle, which can be done thanks to Remark \ref{R:orient}, a complete $2$-bottom graph presentation of $\tilde{T}$ is represented in Figure \ref{F:bottom_graph_presentation}. Then we have
 \[
  F_\intL(\tilde{T}) = \zeta \DD^{-1}.
 \]
 Indeed, this follows from Lemma \ref{L:inverse_Drinfeld} using the definition of the coproduct $\copL$ given by Equation \eqref{E:coalgebra_structure} and of the mirrored pairing $\tilde{\pairL}$ given by Equation \eqref{E:mirrored_pairing}.
\end{proof}

\begin{figure}[ht]
 \centering
 \includegraphics{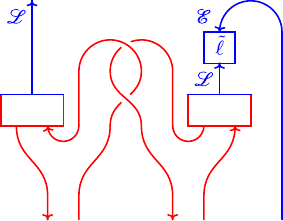}
 \caption{Complete $2$-bottom graph presentation of $\tilde{T}$.}
 \label{F:bottom_graph_presentation}
\end{figure}

\FloatBarrier

\subsection{Skein equivalence}\label{S:skein_equivalence}

We introduce now the concept of skein equivalence for morphisms of $\calR_\intL$. If $(\underline{\epsilon},\underline{V})$ and $(\underline{\epsilon'},\underline{V'})$ are objects of $\calR_\intL$ then we say two formal linear combinations $\sum_{i=1}^m \alpha_i \cdot T_i$ and $\sum_{i'=1}^{m'} \alpha'_{i'} \cdot T'_{i'}$ of morphisms in $\calR_\intL((\underline{\epsilon},\underline{V}),(\underline{\epsilon'},\underline{V'}))$ are \textit{skein equivalent} if
\[
 \sum_{i=1}^m \alpha_i \cdot F_\intL(T_i) = \sum_{i'=1}^{m'} \alpha'_{i'} \cdot F_\intL(T'_{i'}).
\]
Such a skein equivalence will be denoted
\[
 \sum_{i=1}^m \alpha_i \cdot T_i \doteq \sum_{i'=1}^{m'} \alpha'_{i'} \cdot T'_{i'}.
\]

\begin{lemma}[Cutting Lemma]\label{L:cutting}
 For every $V \in \Irr$ we have the skein equivalence
 \begin{equation}\label{E:cutting}
  \raisebox{-0.5\height}{\includegraphics{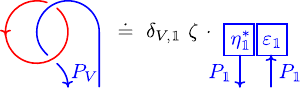}}
 \end{equation}
 where $\zeta \in \Bbbk^*$ is the modularity parameter of Lemma~\ref{lem:coint-from-int}, $\eta_{\one} : \one \to P_{\one}$ is the injection morphism of Lemma~\ref{lem:lambda-iota}, and $\epsilon_{\one} : P_{\one} \to \one$ is the canonical surjection of Section~\ref{S:unimodularity}.
\end{lemma}

\begin{proof}
 If $h_V$ denotes the image under $F_\intL$ of the morphism represented in the left-hand side of Equation \eqref{E:cutting}, then Equation \eqref{E:Hopf_pairing} gives
 \[
  h_V = \pairL \circ (\intL \otimes i_{P_V}).
 \]
 Now Lemma \ref{lem:coint-from-int} implies
 \[
  h_V = \zeta \cointL \circ i_{P_V},
 \]
 and, thanks to Lemma \ref{lem:lambda-iota}, we have
 \[
  h_V = \delta_{V,\one} \zeta \eta_{\one}^* \otimes \epsilon_{\one}. \qedhere
 \]
\end{proof}

\begin{figure}[b]
 \centering
 \includegraphics{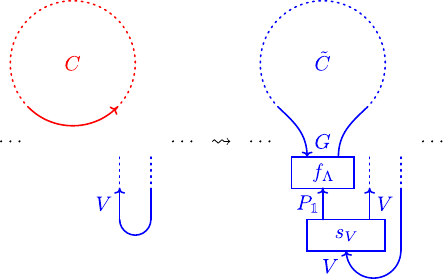}
 \caption{Red-to-blue operation. If $C$ intersects a bichrome coupon, then the resulting blue coupon of $\tilde{C}$ is labeled by either $i_G$ or $j_G$ according to its configuration, see Figure \ref{F:smoothing}.}
 \label{F:red_turns_blue}
\end{figure}

\begin{lemma}\label{L:red_turns_blue}
 If $G$ is a projective generator of $\calC$ and $V \in \Proj(\calC)$, then there exist morphisms $f_\intL \in \calC(P_{\one},G^* \otimes G)$ and $s_V \in \calC(V,P_{\one} \otimes V)$ satisfying 
 \[
  i_G \circ f_\intL = \intL \circ \epsilon_{\one}, \quad 
  (\epsilon_{\one} \otimes \id_V) \circ s_V = \id_V.
 \]
 Furthermore, if $T$ is an admissible morphism of $\calR_\intL$, if $C$ is a red cycle of $T$, and if $\tilde{T}$ denotes the admissible morphism of $\calR_\intL$ obtained by replacing $C$ with the blue graph $\tilde{C}$ represented in Figure \ref{F:red_turns_blue} for any choice of $f_\intL$ and $s_V$ as above, then
 \[
  F_\intL(T) = F_\intL(\tilde{T}).
 \] 
 Similarly, if $T$ is also closed, then $F'_\intL(T) = F'_\intL(\tilde{T})$.
\end{lemma}

We postpone the proof of Lemma \ref{L:red_turns_blue} to Appendix \ref{S:proof_red_turns_blue}. We can now give an alternative proof of the twist non-degeneracy of $\calC$ which also relates the stabilization coefficients $\Delta_\pm$ to the modularity parameter $\zeta$.

\begin{corollary}\label{C:non-degeneracy}
 If $\calC$ is a modular category then $\Delta_- \Delta_+ = \zeta$.
\end{corollary}

\begin{proof}
 
 \begin{figure}[t]
  \centering
   \includegraphics{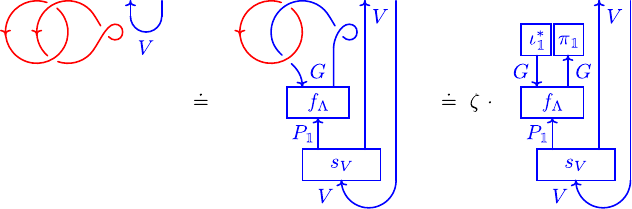}
   \caption{Skein equivalence witnessing $\Delta_- \Delta_+ = \zeta$.}
   \label{F:non-degeneracy}
 \end{figure}

 Let us choose for simplicity the projective generator $G$ determined by Equation \eqref{E:proj_gen} with $n_V = 1$ for all $V \in \Irr$. Then we can decompose $\id_G$ as
 \begin{equation}\label{E:id_G}
  \id_G = \sum_{V \in \Irr} \iota_{P_V} \circ \pi_{P_V}
 \end{equation}
 for some epimorphism $\pi_{P_V} \in \calC(G,P_V)$ and some monomorphism $\iota_{P_V} \in \calC(P_V,G)$ for every $V \in \Irr$. If $V \in \Proj(\calC)$ then, thanks to Equation \eqref{E:id_G} and Lemmas \ref{L:cutting} and \ref{L:red_turns_blue}, we have the sequence of skein equivalences of Figure \ref{F:non-degeneracy}, where
 \begin{equation}\label{E:pi_one_iota_one}
  \pi_{\one} := \epsilon_{\one} \circ \pi_{P_{\one}} \in \calC(G,\one), \qquad
  \iota_{\one} := \iota_{P_{\one}} \circ \eta_{\one} \in \calC(\one,G).
 \end{equation}

 On one hand, $F_\intL$ maps the left-hand side of Figure \ref{F:non-degeneracy} to $\Delta_- \Delta_+ \cdot \lcoev_V$. Indeed, the red link is obtained by sliding a $+1$-framed unknot over a $-1$-framed unknot, and $F_\intL$ is invariant under this operation thanks to Proposition \ref{P:slide}. Then Definition \ref{D:twist_non-degenerate}, together with the definition of the counit $\counitL$ in Equation~\eqref{E:coalgebra_structure}, implies the evaluation of $F_\intL$ against a $\pm1$-framed unknot is equal to $\Delta_\pm$.

 On the other hand, $F_\intL$ maps the right-hand side of Figure \ref{F:non-degeneracy} to $\zeta \cdot \lcoev_V$, because
 \begin{align*}
  (\iota_{\one}^* \otimes \pi_{\one}) \circ f_\intL 
  &= (\eta_{\one}^* \otimes \epsilon_{\one}) \circ (\iota_{P_{\one}}^* \otimes \pi_{P_{\one}}) \circ f_\intL \\*
  &= \cointL \circ i_{P_{\one}} \circ (\iota_{P_{\one}}^* \otimes \pi_{P_{\one}}) \circ f_\intL \\*
  &= \sum_{V \in \Irr} \cointL \circ i_{P_V} \circ (\iota_{P_V}^* \otimes \pi_{P_V}) \circ f_\intL \\*
  &= \sum_{V \in \Irr} \cointL \circ i_G \circ \left(\id_G^* \otimes (\iota_{P_V} \circ \pi_{P_V}) \right) \circ f_\intL \\*
  &= \cointL \circ i_G \circ f_\intL \\*
  &= \cointL \circ \intL \circ \epsilon_{\one} \\*
  &= \epsilon_{\one},
 \end{align*}
 where the first equality follows Equation \eqref{E:pi_one_iota_one}, the second and third ones from Lemma \ref{lem:lambda-iota}, the fourth one from dinaturality of $i$, the fifth one from Equation \eqref{E:id_G}, the sixth one from the definition of $f_\intL$ in Lemma \ref{L:red_turns_blue}, and the last one from Lemma \ref{L:identities} by recalling our assumption $\cointL \circ \intL = \id_{\one}$.
\end{proof}

\FloatBarrier

\subsection{Admissible cobordism category and universal construction}\label{S:admissible_cobordisms}

Following \cite[Sec.~3.3]{DGP17} closely, we first introduce a category $\adCob_{\calC}$ of admissible cobordisms, and then we apply the universal construction of \cite{BHMV95} to obtain a functorial extension of the renormalized Lyubashenko invariant $\rmL'_\calC$ given by Theorem \ref{T:admissible_closed_3-manifold_invariant}. Before starting, we need to extend a few definitions to a more general setting. A \textit{blue set $P$ inside a surface $\Sigma$} is a discrete set of blue points of $\Sigma$ endowed with orientations, framings, and labels given by objects of $\calC$. A \textit{bichrome graph $T$ inside a 3-dimensional cobordism $M$} is a bichrome graph embedded inside $M$ whose boundary vertices are given by blue sets inside the boundary of the cobordism. With this terminology in place, we can define the symmetric monoidal category $\Cob_{\calC}$.

An \textit{object $\bbSigma$ of $\Cob_{\calC}$} is a triple $(\Sigma,P,\lambda)$ where: 
\begin{enumerate}
 \item $\Sigma$ is a closed surface;
 \item $P \subset \Sigma$ is a blue set;
 \item $\lambda \subset H_1(\Sigma;\R)$ is a Lagrangian subspace.
\end{enumerate}

A \textit{morphism $\bbM : \bbSigma \rightarrow \bbSigma'$ of $\Cob_{\calC}$} is an equivalence class of triples $(M,T,n)$ where:
\begin{enumerate}
 \item $M$ is a 3-dimensional cobordism from $\Sigma$ to $\Sigma'$;
 \item $T \subset M$ is a bichrome graph from $P$ to $P'$;
 \item $n \in \Z$ is an integer called the \textit{signature defect}.
\end{enumerate}
Two triples $(M,T,n)$ and $(M',T',n')$ are equivalent if $n = n'$ and if there exists an isomorphism of cobordisms $f : M \rightarrow M'$ satisfying $f(T) = T'$.

The \textit{identity morphism $\id_{\bbSigma} : \bbSigma \rightarrow \bbSigma$ associated with an object $\bbSigma = (\Sigma,P,\lambda)$ of $\Cob_{\calC}$} is the equivalence class of the triple
\[
 (\Sigma \times I,P \times I,0).
\]

The \textit{composition $\bbM' \circ \bbM : \bbSigma \rightarrow \bbSigma''$ of morphisms $\bbM' : \bbSigma' \rightarrow \bbSigma''$, $\bbM : \bbSigma \rightarrow \bbSigma'$ of $\Cob_{\calC}$} is the equivalence class of the triple
\[
 \left( M \cup_{\Sigma'} M', T \cup_{P'} T', n + n' - \mu(M_*(\lambda),\lambda',M'^* (\lambda'')) \right)
\]
for the Lagrangian subspaces 
\begin{gather*}
 M_*(\lambda) := \{ x' \in H_1(\Sigma';\R) \mid (i_{M_+})_*(x') \in (i_{M_-})_* (\lambda) \} \subset H_1(\Sigma';\R) \\*
 M'^*(\lambda'') := \{ x' \in H_1(\Sigma';\R) \mid (i_{M'_-})_*(x') \in (i_{M'_+})_*(\lambda'') \} \subset H_1(\Sigma';\R).
\end{gather*}
where
\[
 i_{M_-} : \Sigma \hookrightarrow M, \quad i_{M_+} : \Sigma' \hookrightarrow M, \quad
 i_{M'_-} : \Sigma' \hookrightarrow M', \quad i_{M'_+} : \Sigma'' \hookrightarrow M'
\]
are the embeddings induced by the structure maps of $M$ and $M'$. Here $\mu$ denotes the Maslov index, see \cite{T94} for a detailed account of its properties.

The \textit{unit of $\Cob_{\calC}$} is the unique object whose surface is empty, and it will be denoted $\varnothing$. The \textit{tensor product $\bbSigma \disjun \bbSigma'$ of objects $\bbSigma$, $\bbSigma'$ of $\Cob_{\calC}$} is the triple 
\[
 (\Sigma \sqcup \Sigma',P \sqcup P',\lambda \oplus \lambda').
\]
The \textit{tensor product $\bbM \disjun \bbM' : \bbSigma \disjun \bbSigma' \rightarrow \bbSigma'' \disjun \bbSigma'''$ of morphisms $\bbM : \bbSigma \rightarrow \bbSigma''$, $\bbM' : \bbSigma' \rightarrow \bbSigma'''$ of $\Cob_{\calC}$} is the equivalence class of the triple
\[
 (M \sqcup M',T \sqcup T',n+n').
\]
It is straightforward to define dualities and trivial braidings which make $\Cob_\calC$ into a rigid symmetric monoidal category.

We will now construct a functor extending the renormalized Lyubashenko invariant $\rmL'_\calC$. Its domain however will not be the whole symmetric monoidal category $\Cob_{\calC}$, as there is no way of defining $\rmL'_\calC$ for every closed morphism of $\Cob_{\calC}$. Indeed, we will have to consider a strictly smaller subcategory.

\begin{definition}\label{D:adm_cob_cat}
 The \textit{admissible cobordism category} $\adCob_{\calC}$ is the symmetric mon\-oid\-al subcategory of $\Cob_{\calC}$ having the same objects but featuring only morphisms $\bbM = (M,T,n)$ which satisfy the following \textit{admissibility condition}:
\begin{quote}	
Every connected component of $M$ disjoint from the incoming boundary contains an admissible bichrome subgraph of $T$.
\end{quote}
\end{definition}

We can now extend the renormalized Lyubashenko invariant to closed morphisms of $\adCob_{\calC}$ by setting
\[
 \rmL'_\calC(\bbM) := \delta^n \rmL'_\calC(M,T)
\]
for every closed connected morphism $\bbM = (M,T,n)$, where $\delta \in \Bbbk^*$ is the coefficient fixed in Theorem~\ref{T:admissible_closed_3-manifold_invariant}, and then by setting
\[
 \rmL'_\calC(\bbM_1 \disjun \ldots \disjun \bbM_k) := \prod_{i=1}^k \rmL'_\calC(\bbM_i)
\]
for every tensor product of closed connected morphisms $\bbM_1, \ldots, \bbM_k$.

We apply now the universal construction of \cite{BHMV95}, which allows us to obtain a functorial extension of $\rmL'_\calC$. If $\bbSigma$ is an object of $\adCob_{\calC}$ then let $\calV(\bbSigma)$ be the free vector space generated by the set of morphisms $\bbM_{\bbSigma} : \varnothing \rightarrow \bbSigma$ of $\adCob_{\calC}$, and let $\calV'(\bbSigma)$ be the free vector space generated by the set of morphisms $\bbM'_{\bbSigma} : \bbSigma \rightarrow \varnothing$ of $\adCob_{\calC}$. Next, consider the bilinear form
\[
 \begin{array}{rccc}
  \langle \cdot , \cdot \rangle_{\bbSigma} : & \calV'(\bbSigma) \times \calV(\bbSigma) & \rightarrow & \Bbbk \\
  & (\bbM'_{\bbSigma},\bbM_{\bbSigma}) & \mapsto & \rmL'_\calC(\bbM'_{\bbSigma} \circ \bbM_{\bbSigma}).
 \end{array}
\]
Let $\rmV_{\calC}(\bbSigma)$ be the quotient of the vector space $\calV(\bbSigma)$ with respect to the right radical of the bilinear form $\langle \cdot , \cdot \rangle_{\bbSigma}$, and similarly let $\rmV'_{\calC}(\bbSigma)$ be the quotient of the vector space $\calV'(\bbSigma)$ with respect to the left radical of the bilinear form $\langle \cdot , \cdot \rangle_{\bbSigma}$. We will use the notation $[ {} \cdot {} ] : \calV(\bbSigma) \to \rmV_\calC(\bbSigma)$ and $[ {} \cdot {} ] : \calV'(\bbSigma) \to \rmV'_\calC(\bbSigma)$ for both projections. Note that the pairing $\langle \cdot , \cdot \rangle_{\bbSigma}$ induces a non-degenerate pairing 
\[
 \langle \cdot , \cdot \rangle_{\bbSigma} :  \rmV'_{\calC}(\bbSigma) \otimes \rmV_{\calC}(\bbSigma) \rightarrow \Bbbk. 
\]

Now if $\bbM : \bbSigma \rightarrow \bbSigma'$ is a morphism of $\adCob_{\calC}$, then let $\rmV_{\calC}(\bbM)$ be the linear map defined by
\[
 \begin{array}{rccc}
  \rmV_{\calC}(\bbM) : & \rmV_{\calC}(\bbSigma) & \rightarrow & \rmV_{\calC}(\bbSigma') \\
  & [\bbM_{\bbSigma}] & \mapsto & [\bbM \circ \bbM_{\bbSigma}],
 \end{array}
\]
and similarly let $\rmV'_{\calC}(\bbM)$ be the linear map defined by
\[
 \begin{array}{rccc}
  \rmV'_{\calC}(\bbM) : & \rmV'_{\calC}(\bbSigma') & \rightarrow & \rmV'_{\calC}(\bbSigma) \\
   & [\bbM'_{\bbSigma'}] & \mapsto & [\bbM'_{\bbSigma'} \circ \bbM].
 \end{array}
\]

The construction we just provided clearly defines functors
\[
 \rmV_{\calC} : \adCob_{\calC} \rightarrow \Vect_{\Bbbk}, \quad \rmV'_{\calC} : (\adCob_{\calC})^{\op} \rightarrow \Vect_{\Bbbk}.
\]

\begin{proposition}\label{lax_monoidality}
 The natural transformation $\mu : \otimes \circ (\rmV_\calC \times \rmV_\calC) \Rightarrow \rmV_\calC \circ \disjun$  associating with every pair of objects $\bbSigma$, $\bbSigma'$ of $\adCob_{\calC}$ the linear map 
 \[
  \begin{array}{rccc}
   \mu_{\bbSigma,\bbSigma'} : & \rmV_\calC(\bbSigma) \otimes \rmV_\calC(\bbSigma') & \rightarrow & 
   \rmV_\calC(\bbSigma \disjun \bbSigma') \\
   & [\bbM_{\bbSigma}] \otimes [\bbM_{\bbSigma'}] & \mapsto & [\bbM_{\bbSigma} \disjun \bbM_{\bbSigma'}]
  \end{array}
 \]
 is a monomorphism.
\end{proposition}

\begin{proof}
 First of all, $\mu_{\bbSigma,\bbSigma'}$ is well-defined. Indeed, if $\sum_{i=1}^m \alpha_i \cdot [\bbM_{\bbSigma,i}]$ is a trivial vector in $\rmV_\calC(\bbSigma)$, then $\sum_{i=1}^m \alpha_i \cdot [\bbM_{\bbSigma,i} \disjun \bbM_{\bbSigma'}]$ is a trivial vector in $\rmV_\calC(\bbSigma \disjun \bbSigma')$ for every $\bbM_{\bbSigma'} \in \calV(\bbSigma')$, because for every $\bbM'_{\bbSigma \disjun \bbSigma'} \in \calV'(\bbSigma \disjun \bbSigma')$ we have
 \[
  \sum_{i=1}^m \alpha_i \rmL'_\calC(\bbM'_{\bbSigma \disjun \bbSigma'} \circ (\bbM_{\bbSigma,i} \disjun \bbM_{\bbSigma'})) = \sum_{i=1}^m \alpha_i \rmL'_\calC(\bbM'_{\bbSigma} \circ \bbM_{\bbSigma,i}) = 0,
 \]
 where $\bbM'_{\bbSigma} = \bbM'_{\bbSigma \disjun \bbSigma'} \circ (\id_{\bbSigma} \disjun \bbM_{\bbSigma'})$. The same holds when switching the roles of $\bbSigma$ and $\bbSigma'$. Next, $\mu_{\bbSigma,\bbSigma'}$ is natural. Indeed, this follows immediately from
 \[
  (\bbM \circ \bbM_{\bbSigma}) {}\disjun{} (\bbM' \circ \bbM_{\bbSigma'}) = (\bbM \disjun \bbM') \circ (\bbM_{\bbSigma} \disjun \bbM_{\bbSigma'}),
 \]
 which holds for all $\bbM_{\bbSigma} \in \calV(\bbSigma)$ and $\bbM_{\bbSigma'} \in \calV(\bbSigma')$, and for all $\bbM : \bbSigma \to \bbSigma''$ and $\bbM' : \bbSigma' \to \bbSigma'''$. Finally, $\mu_{\bbSigma,\bbSigma'}$ is injective. Indeed, if $\sum_{i=1}^m \alpha_i \cdot [\bbM_{\bbSigma,i} \disjun \bbM_{\bbSigma',i}]$ is a trivial vector in $\rmV_\calC(\bbSigma \disjun \bbSigma')$, then it satisfies
 \[
  \sum_{i=1}^m \alpha_i \rmL'_\calC(\bbM'_{\bbSigma \disjun \bbSigma'} \circ (\bbM_{\bbSigma,i} \disjun \bbM_{\bbSigma',i})) = 0
 \]
 for every $\bbM'_{\bbSigma \disjun \bbSigma'} \in \calV'(\bbSigma \disjun \bbSigma')$. In particular its pairing with every vector of the form $\bbM'_{\bbSigma} \disjun \bbM'_{\bbSigma'} \in \calV'(\bbSigma \disjun \bbSigma')$ for some $\bbM'_{\bbSigma} \in \calV'(\bbSigma)$ and $\bbM'_{\bbSigma'} \in \calV'(\bbSigma')$ must be zero too. This means $\sum_{i=1}^m \alpha_i \cdot [\bbM_{\bbSigma,i}] \otimes [\bbM_{\bbSigma',i}]$ is a trivial vector in ${\rmV_\calC(\bbSigma) \otimes \rmV_\calC(\bbSigma')}$.
\end{proof}

\begin{remark}\label{R:TQFT_ssi}
 Thanks to Remark \ref{R:3-man_inv_ssi}, when the category $\calC$ is semisimple $\rmV_\calC$ is a TQFT, and it is precisely the standard Reshetikhin-Turaev one.
\end{remark}

\subsection{Surgery axioms}\label{S:surgery_axioms}

We move on to study the behavior of $\rmL'_\calC$ under decorated index $k$ surgery for $k \in \{ 0,1,2 \}$. This is a crucial step in the proof of the symmetric monoidality of a functor produced by the universal construction, and the strategy dates back to \cite{BHMV95}. The specific version of these topological operations we need here was introduced in \cite[Sec.~3.4]{DGP17} in the context of factorizable Hopf algebras, but everything can be directly adapted to our setting. Indeed, for every $k \in \{ 0,1,2 \}$, we can consider an object $\bbSigma_k$ of $\Cob_\calC$ called the \textit{index $k$ surgery surface}, and morphisms $\bbA_k,\bbB_k : \varnothing \rightarrow \bbSigma_k$ of $\Cob_\calC$ called the \textit{index $k$ attaching tube} and the \textit{index $k$ belt tube} respectively. The closed surface $\varSigma_k$ of $\bbSigma_k$ is given by $S^{k-1} \times S^{3-k}$ with the convention $S^{-1} := \varnothing$. The cobordisms $A_k$ and $B_k$ of $\bbA_k$ and $\bbB_k$ are given by $S^{k-1} \times (-1)^{k-1} D^{4-k}$ and $D^k \times S^{3-k}$ respectively, with the convention $D^{0} := \{ 0 \}$. The blue set $P_{\Sigma_1}$ of $\bbSigma_1$ is given by $S^0 \times \{ (0,0,1) \}$ with orientation induced by $S^0$ and label $P_{\one}$, while all the other blue sets are empty. The blue graph $T_{B_0}$ of $\bbB_0$ coincides with the one represented in Equation \eqref{E:renormalized_invariant_example}; The blue graph $T_{A_1}$ of $\bbA_1$ is obtained from $T_{B_0}$ by cutting it along its only edge, and by embedding the top half into $\{ -1 \} \times \overline{D^3}$ and the bottom half into $\{ 1 \} \times D^3$, compare with \cite[Fig.~26]{DGP17}; The blue tangle $T_{B_1}$ of $\bbB_1$ is given by the edge $D^1 \times \{ (0,0,1) \}$, with orientation and label determined by $P_{\Sigma_1}$; The red knot $K_{A_2}$ of $\bbA_2$ is given by the core $S^1 \times \{ (0, 0) \}$; All the other bichrome graphs are empty. Lagrangians and signature defects coincide with those of \cite[Sec.~3.4]{DGP17}. Then, for $k \in \{0,1,2\}$ and for a morphism $\bbM_k : \bbSigma_k \rightarrow \varnothing$ of $\adCob_{\calC}$, the morphism $\bbM_k \circ \bbB_k$ is said to be obtained from $\bbM_k \circ \bbA_k$ by an \emph{index $k$ surgery}. The next statement describes the behavior of $\rmL'_\calC$ under this operation (compare with \cite[Prop.~3.12]{DGP17}), and its proof will occupy the remainder of this section.

\begin{proposition}\label{P:surgery_axioms}
 For $k \in \{0,1,2\}$ let $\bbM_k : \bbSigma_k \rightarrow \varnothing$ be a morphism of $\adCob_{\calC}$. If $\bbM_k \circ \bbA_k$ is in $\adCob_{\calC}$ then 
 \[
  \rmL'_\calC(\bbM_k \circ \bbB_k) = \alpha_k \rmL'_\calC(\bbM_k \circ \bbA_k)
 \]
 with $\alpha_0 = \alpha_1^{-1} = \alpha_2 = \calD^{-1}$.
\end{proposition}

\begin{proof}
 If $k = 0$ then the property reduces to the computation
 \[
  \rmL'_\calC(\bbB_0) = \calD^{-1 - 0} \delta^{0 - 0} F'_\intL(T_{B_0}) = \calD^{-1} \rmt_{P_{\one}}(\eta_{\one} \circ \epsilon_{\one}) = \calD^{-1},
 \]
 where we use the normalization of $\rmt$ fixed at the beginning of Section \ref{S:TQFTs}.

 If $k = 1$ then we have two cases, according to whether or not the surgery involves two different connected components of the closed morphism $\bbM_1 \circ \bbA_1$. Let us start from the first case, and let us begin by decomposing $\bbSigma_1$ as a tensor product $\smash{\overline{\bbS^2_{(-,P_{\one})}}} \disjun \bbS^2_{(+,P_{\one})}$, where 
 \[
  \overline{\bbS^2_{(-,P_{\one})}} = \left( \overline{S^2},P_{(-,P_{\one})},\{ 0 \} \right), \qquad
  \bbS^2_{(+,P_{\one})} = \left( S^2,P_{(+,P_{\one})},\{ 0 \} \right),
 \]
 with $P_{(-,P_{\one})}$ and $P_{(+,P_{\one})}$ both featuring a single blue point with orientation and label specified by subscripts. Next, let us decompose $\bbA_1$ as a tensor product $\smash{\overline{\bbD^3_{\epsilon_{\one}}}} \disjun \bbD^3_{\eta_{\one}}$ with respect to morphisms $\smash{\overline{\bbD^3_{\epsilon_{\one}}} : \varnothing \rightarrow \overline{\bbS^2_{(-,P_{\one})}}}$ and $\bbD^3_{\eta_{\one}} :  \varnothing \rightarrow \bbS^2_{(+,P_{\one})}$, where 
 \[
  \overline{\bbD^3_{\epsilon_{\one}}} = \left( \overline{D^3},T_{\epsilon_{\one}},0 \right), \qquad
  \bbD^3_{\eta_{\one}} = \left( D^3,T_{\eta_{\one}},0 \right),
 \]  
 with $T_{\epsilon_{\one}}$ and $T_{\eta_{\one}}$ both featuring a single blue coupon with label specified by subscripts. Then, let us consider connected morphisms $\bbM_1 : \smash{\overline{\bbS^2_{(-,P_{\one})}}} \rightarrow \varnothing$ and $\bbM'_1 : \bbS^2_{(+,P_{\one})} \rightarrow \varnothing$ of $\adCob_{\calC}$. If $\bbM_1 = (M_1,T,n)$ and $\bbM'_1 = (M'_1,T',n')$ then
 \begin{gather*}
  \bbM_1 \circ \overline{\bbD^3_{\epsilon_{\one}}} = \left( \overline{D^3} \cup_{\overline{S^2}} M_1,T_{\epsilon_{\one}} \cup_{\overline{P_1}} T,n \right), \\*
  \bbM'_1 \circ \bbD^3_{\eta_{\one}} = \left( D^3 \cup_{S^2} M'_1,T_{\eta_{\one}} \cup_{P_1} T',n \right).
 \end{gather*}
 Let us set for convenience $\hat{T} := T_{\epsilon_{\one}} \cup_{\overline{P_1}} T$ and $\hat{T}' := T_{\eta_{\one}} \cup_{P_1} T'$. Then, if $L$ is an $\ell$-component surgery link for $\smash{M_1 \cup_{\overline{S^2}} \overline{D^3}}$, and if $L'$ is an $\ell'$-component surgery link for $M'_1 \cup_{S^2} D^3$, we have
 \begin{gather*}
  \rmL'_\calC(\bbM_1 \circ \overline{\bbD^3_{\epsilon_{\one}}}) = \calD^{-1-\ell} \delta^{n - \sigma(L)} F'_\intL \left( L \cup \hat{T} \right), \\*
  \rmL'_\calC(\bbM'_1 \circ \bbD^3_{\eta_{\one}}) = \calD^{-1-\ell'} \delta^{n' - \sigma(L')} F'_\intL \left( L' \cup \hat{T}' \right).
 \end{gather*}
 Now let us use the same notation $T_{\epsilon_{\one}}$ and $T_{\eta_{\one}}$ also for the unique morphisms ${T_{\epsilon_{\one}} : (+,P_{\one}) \rightarrow \varnothing}$ and $T_{\eta_{\one}} : \varnothing \rightarrow (+,P_{\one})$ of $\calR_\intL$ determined by the decorations of $\smash{\overline{\bbD^3_{\epsilon_{\one}}}}$ and $\bbD^3_{\eta_{\one}}$. This determines uniquely morphisms $L \cup T : \varnothing \rightarrow (+,P_{\one})$ and $L' \cup T' : \varnothing \rightarrow (+,P_{\one})$ of $\calR_\intL$ satisfying
 \[
  L \cup \hat{T} = T_{\epsilon_{\one}} \circ (L \cup T), \qquad
  L' \cup \hat{T}' = (L' \cup T') \circ T_{\eta_{\one}},
 \]
 or graphically
 \[
  \raisebox{-0.5\height}{\includegraphics{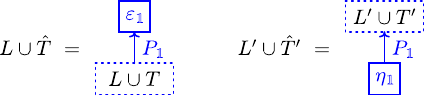}}
 \]
 Then $(L \cup T) \circ T_{\epsilon_{\one}}$ is a cutting presentation of $L \cup \hat{T}$, and $T_{\eta_{\one}} \circ (L' \cup T')$ is a cutting presentation of $L' \cup \hat{T}'$. This means
 \begin{gather*}
  F'_\intL \left( L \cup \hat{T} \right) = \rmt_{P_{\one}} \left( F_\intL \left( (L \cup T) \circ T_{\epsilon_{\one}} \right) \right), \\*
  F'_\intL \left( L' \cup \hat{T}' \right) = \rmt_{P_{\one}} \left( F_\intL \left( T_{\eta_{\one}} \circ (L' \cup T') \right) \right).
 \end{gather*}
 Furthermore,
 \begin{gather*}
  F_\intL \left( L \cup T \right) = \rmt_{P_{\one}} \left( F_\intL \left( (L \cup T) \circ T_{\epsilon_{\one}} \right) \right) \cdot \eta_{\one}, \\*
  F_\intL \left( L' \cup T' \right) = \rmt_{P_{\one}} \left( F_\intL \left( T_{\eta_{\one}} \circ (L' \cup T') \right) \right) \cdot \epsilon_{\one},
 \end{gather*}
 because $\calC(\one,P_{\one})$ and $\calC(P_{\one},\one)$ are 1-dimensional, generated by $\eta_{\one}$ and $\epsilon_{\one}$ respectively, and because $\rmt_{P_{\one}}(\eta_{\one} \circ \epsilon_{\one}) = 1$. This means
 \begin{align*}
  F'_\intL &\left( (L' \cup T') \circ (L \cup T) \right) \\*
  &= \rmt_{P_{\one}} \left( F_\intL \left( (L \cup T) \circ (L' \cup T') \right) \right) \\*
  &= \rmt_{P_{\one}} \left( F_\intL \left( (L \cup T) \circ T_{\epsilon_{\one}} \right) \right)
  \rmt_{P_{\one}} \left( F_\intL \left( T_{\eta_{\one}} \circ (L' \cup T') \right) \right)
  \rmt_{P_{\one}} \left( \eta_{\one} \circ \epsilon_{\one} \right) \\*
  &= F'_\intL ( L \cup \hat{T} ) F'_\intL ( L' \cup \hat{T}' ).
 \end{align*}
 But now
 \[
  \rmL'_\calC((\bbM_1 \disjun \bbM'_1) \circ \bbB_1) = \calD^{-1-\ell-\ell'} \delta^{n + n' - \sigma(L) - \sigma(L')} 
  F'_\intL \left( (L' \cup T') \circ (L \cup T) \right).
 \]
 This means
 \begin{align*}
  \rmL'_\calC((\bbM_1 \disjun \bbM'_1) \circ \bbB_1) &= \calD^{-1-\ell-\ell'} \delta^{n + n' - \sigma(L) - \sigma(L')} 
  F'_\intL( L \cup \hat{T} ) \ F'_\intL( L' \cup \hat{T}' ) \\*
  &= \calD \rmL'_\calC(\bbM_1 \circ \overline{\bbD^3_{\epsilon_{\one}}}) \rmL'_\calC(\bbM'_1 \circ \bbD^3_{\eta_{\one}}) \\*
  &= \calD \rmL'_\calC((\bbM_1 \disjun \bbM'_1) \circ \bbA_1).
 \end{align*}

 Now let us move on to the second case, and let us consider a connected morphism $\bbM_1 : \bbSigma_1 \rightarrow \varnothing$ of $\adCob_{\calC}$. If $\bbM_1 = (M_1,T,n)$ then
 \[
  \bbM_1 \circ \bbA_1 = \left( (S^{0} \times D^{3}) \cup_{(S^{0} \times S^{2})} M_1,T_{A_1} \cup_{P_{\Sigma_1}} T,n \right).
 \]
 Let us set for convenience $\hat{T} := T_{A_1} \cup_{P_{\Sigma_1}} T$. Then, if $L$ is an $\ell$-component surgery link for $(S^{0} \times D^{3}) \cup_{(S^{0} \times S^{2})} M_1$, we have
 \[
  \rmL'_\calC(\bbM_1 \circ \bbA_1) = \calD^{-1-\ell} \delta^{n - \sigma(L)} F'_\intL( L \cup \hat{T} ).
 \]
 Let $L \cup T : (+,P_{\one}) \rightarrow (+,P_{\one})$ be the unique morphism of $\calR_\intL$ satisfying 
 \[
  L \cup \hat{T} = T_{\epsilon_{\one}} \circ (L \cup T) \circ T_{\eta_{\one}}.
 \]
 Then $T_{\eta_{\one}} \circ T_{\epsilon_{\one}} \circ (L \cup T)$ is a cutting presentation of $L \cup \hat{T}$. This means
 \[
  F'_\intL( L \cup \hat{T} ) = \rmt_{P_{\one}} \left( F_\intL \left( T_{\eta_{\one}} \circ T_{\epsilon_{\one}} \circ (L \cup T) \right) \right).
 \]
 Now let $L' \cup T'$ denote the admissible bichrome graph
 \[
  \raisebox{-0.5\height}{\includegraphics{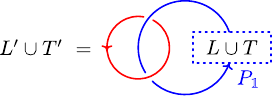}}
 \]
 Thanks to Lemma \ref{L:cutting}, we have
 \[
  F'_\intL \left( L' \cup T' \right) = \zeta F'_\intL( L \cup \hat{T} ).
 \]
 But now
 \[
  \rmL'_\calC(\bbM_1 \circ \bbB_1) = \calD^{-1-(\ell+1)} \delta^{n - \sigma(L')} F'_\intL (L' \cup T').
 \]
 Then, since $\sigma(L') = \sigma(L)$ and $\zeta = \calD^2$, we have
 \begin{align*}
  \rmL'_\calC(\bbM_1 \circ \bbB_1) &= \zeta \calD^{-2-\ell} \delta^{n - \sigma(L)} F'_\intL( L \cup \hat{T} ) \\*
  &= \calD^{-\ell} \delta^{n - \sigma(L)} F'_\intL( L \cup \hat{T} ) \\*
  &= \calD \rmL'_\calC( \bbM_1 \circ \bbA_1 ).
 \end{align*}

 If $k = 2$ let us consider a connected morphism $\bbM_2 : \bbSigma_2 \rightarrow \varnothing$ of $\adCob_{\calC}$. If $\bbM_2 = (M_2,T,n)$ then
 \begin{gather*}
  \bbM_2 \circ \bbA_2 = \left( (S^{1} \times \overline{D^{2}}) \cup_{(S^1 \times S^1)} M_2,K_{A_2} \cup T,n \right), \\*
  \bbM_2 \circ \bbB_2 = \left( (D^{2} \times S^{1}) \cup_{(S^1 \times S^1)} M_2,T,n + \sigma(L \cup K_{A_2}) - \sigma(L) \right).
 \end{gather*}
 If $L$ is an $\ell$-component surgery link for $(S^{1} \times \overline{D^{2}}) \cup_{(S^1 \times S^1)} M_2$, then $L \cup K_{A_2}$ is an $\ell + 1$-component surgery link for $(D^{2} \times S^{1}) \cup_{(S^1 \times S^1)} M_2$. Therefore
 \begin{gather*}
  \rmL'_\calC(\bbM_2 \circ \bbA_2) = \calD^{-1 - \ell} \delta^{n - \sigma(L)} F'_\intL(L \cup K_{A_2}), \\*
  \rmL'_\calC(\bbM_2 \circ \bbB_2) = \calD^{-1 - (\ell + 1)} \delta^{(n + \sigma(L \cup K_2) - \sigma(L)) -   \sigma(L \cup K_{A_2})} F'_\intL(L \cup K_{A_2}).
 \end{gather*}
 Note that the red knot $K_{A_2}$ plays the role of a decoration for $\bbM_2 \circ \bbA_2$, while it plays the role of a surgery component for $\bbM_2 \circ \bbB_2$. This implies 
 \[
  \rmL'_\calC(\bbM_2 \circ \bbB_2) = \calD^{-1} \rmL'_\calC(\bbM_2 \circ \bbA_2). \qedhere
 \]
\end{proof}

\FloatBarrier

\subsection{Connectedness}\label{S:connectedness}

We establish now some useful properties of the functors $\rmV_\calC$ and $\rmV'_\calC$ which will be used for the proof of their monoidality and for the computation of their image. Loosely speaking, they can be summarized as follows:
 \begin{enumerate}
 \item If $\bbSigma$ is an object of $\adCob_\calC$, then $\rmV_\calC(\bbSigma)$ is generated by admissible graphs inside a fixed connected cobordism;
 \item If $\bbSigma$ is a non-empty connected object of $\adCob_\calC$, then $\rmV'_\calC(\bbSigma)$ is generated by graphs inside a fixed connected cobordism.
\end{enumerate}

We start by generalizing the notion of skein equivalence we gave in Section \ref{S:skein_equivalence} for linear combinations of morphisms of $\calR_\intL$ to linear combinations of bichrome graphs inside 3-dimensional cobordisms. This generalized notion of skein equivalence will subtly depend on the cobordism it takes place in, as a result of our definition of the category $\adCob_\calC$. Indeed, if $M$ is a connected cobordism with empty incoming boundary, then bichrome graphs inside it are required to be admissible, and skein equivalences need to preserve this property. Very roughly speaking, in this case we say two linear combinations of admissible bichrome graphs are skein equivalent if they are related by a skein equivalence of $\calR_\intL$ within a 3-ball whose complement contains a projective edge, which is an edge whose label is projective. On the other hand, if $M'$ is a connected cobordism with non-empty incoming boundary, then bichrome graphs inside it need not be admissible, and skein equivalence can be defined as a local relation. In this case, we simply say two linear combinations of bichrome graphs are skein equivalent if they are related by a skein equivalence of $\calR_\intL$ within any 3-ball, regardless of its complement.

In order to give a precise definition, we first need to fix some notation. In particular, we need to specify how to translate bichrome graphs contained inside a 3-ball into morphisms of $\calR_\intL$. In order to do this, let us consider, for every integer $k \geq 0$, an embedding $f_k : D^3 \hookrightarrow \R^2 \times I$ which maps uniformly the 1-dimensional submanifold 
\[
 \{ (x,y,z) \in S^2 \mid y=0, z \geq 0 \} \subset D^3
\] 
to the interval 
\[
 ([0,k+1] \times \{ 0 \}) \times \{ 1 \} \subset \R^2 \times I.
\]
In other words, $f_k$ should map the point $(\cos(\frac{t}{k+1} \pi),0,\sin(\frac{t}{k+1} \pi)) \in D^3$ to the point $((t,0),1) \in \R^2 \times I$ for every $t \in [0,k+1]$. Then, every time we have an object $(\underline{\epsilon},\underline{V}) = ((\epsilon_1,V_1),\ldots,(\epsilon_k,V_k))$ of $\calR_\intL$, we can use the embedding $f_k$ to define by pull back a standard blue set $P_{(\underline{\epsilon},\underline{V})}$ inside $S^2$. We also denote with $f'_k : \overline{D^3} \hookrightarrow \R^2 \times I$ the embedding obtained from $\overline{f_k} : \overline{D^3} \hookrightarrow \overline{\R^2 \times I}$ by composition with the map $\tau : \overline{\R^2 \times I} \to \R^2 \times I$ sending every $((x,y),t) \in \overline{\R^2 \times I}$ to $((x,y),1-t) \in \R^2 \times I$.

Now let us consider an object $\bbSigma = (\Sigma,P,\lambda)$ of $\adCob_\calC$. If $M$ is a connected 3-dimensional cobordism from $\varnothing$ to $\Sigma$, then let us fix an isomorphism of cobordisms $f_M : M \rightarrow D^3 \cup_{S^2} \hat{M}$ for some cobordism $\hat{M}$ from $S^2$ to $\Sigma$ (the choice of which isomorphism $f_M$ to fix is completely inconsequential). We say two linear combinations of admissible bichrome graphs in $M$ from $\varnothing$ to $P$ are \textit{skein equivalent} if, up to isotopies fixing $P$, their images under $f_M$ are of the form $\sum_{i=1}^m \alpha_i \cdot ( T_i \cup \hat{T} )$ and $\sum_{i'=1}^{m'} \alpha'_{i'} \cdot ( T'_{i'} \cup \hat{T})$ for some object $(\underline{\epsilon},\underline{V})$ of $\calR_\intL$, for some linear combinations $\sum_{i=1}^m \alpha_i \cdot T_i$ and $\sum_{i'=1}^{m'} \alpha'_{i'} \cdot T'_{i'}$ of bichrome graphs in $D^3$ from $\varnothing$ to $P_{(\underline{\epsilon},\underline{V})}$ satisfying
\[
 \sum_{i=1}^m \alpha_i \cdot f_k(T_i) \doteq \sum_{i'=1}^{m'} \alpha'_{i'} \cdot f_k(T'_{i'})
\]
in $\calR_\intL(\varnothing,(\underline{\epsilon},\underline{V}))$ as in Section \ref{S:skein_equivalence}, and for some admissible bichrome graph $\hat{T}$ in $\hat{M}$ from $P_{(\underline{\epsilon},\underline{V})}$ to $P$. Skein equivalences inside $M$ will still be denoted by ${}\doteq{}$.

Next, let us suppose the object $\bbSigma$ is non-empty. If $M'$ is a connected 3-dimensional cobordism from $\Sigma$ to $\varnothing$, then let us fix an isomorphism of cobordisms $f_{M'} : M' \rightarrow \hat{M}' \cup_{S^2} \overline{D^3}$ for some cobordism $\hat{M}'$ from $\Sigma$ to $S^2$. We say two linear combinations of bichrome graphs in $M'$ from $P$ to $\varnothing$ are \textit{skein equivalent} if, up to isotopy, their images under $f_{M'}$ are of the form $\sum_{i=1}^m \alpha_i \cdot ( \hat{T}' \cup T_i )$ and $\sum_{i'=1}^{m'} \alpha'_{i'} \cdot ( \hat{T}' \cup T'_{i'} )$ for some object $(\underline{\epsilon},\underline{V})$ of $\calR_\intL$, for some linear combinations $\sum_{i=1}^m \alpha_i \cdot T_i$ and $\sum_{i'=1}^{m'} \alpha'_{i'} \cdot T'_{i'}$ of bichrome graphs in $\overline{D^3}$ from $P_{(\underline{\epsilon},\underline{V})}$ to $\varnothing$ satisfying
\[
 \sum_{i=1}^m \alpha_i \cdot f'_k(T_i) \doteq \sum_{i'=1}^{m'} \alpha'_{i'} \cdot f'_k(T'_{i'})
\]
in $\calR_\intL(\varnothing,(\underline{\epsilon},\underline{V}))$ as in Section \ref{S:skein_equivalence}, and for some bichrome graph $\hat{T}'$ in $\hat{M}'$ from $P$ to $P_{(\underline{\epsilon},\underline{V})}$. As before, skein equivalences inside $M'$ will still be denoted by ${}\doteq{}$.

Let us also quickly observe that the red-to-blue operation defined in Figure \ref{F:red_turns_blue} for morphisms of $\calR_\intL$ can be straightforwardly generalized to admissible bichrome graphs inside connected 3-dimensional cobordisms. Remark however that this is operation is not local, meaning it does not take place inside a 3-ball.

Now, if $\bbSigma = (\Sigma,P,\lambda)$ is an object of $\adCob_\calC$ and $M$ is a connected 3-dimensional cobordism from $\varnothing$ to $\Sigma$, then we denote with $\calV(M;\bbSigma)$ the vector space generated by isotopy classes of admissible bichrome graphs inside $M$ from $\varnothing$ to $P$. Similarly, if $\bbSigma$ is non-empty and $M'$ is a connected 3-dimensional cobordism from $\Sigma$ to $\varnothing$, then we denote with $\calV'(M';\bbSigma)$ the vector space generated by isotopy classes of bichrome graphs inside $M'$ from $P$ to $\varnothing$. We also denote with $\pi_{\bbSigma}$ and with $\pi'_{\bbSigma}$ the natural linear maps
\[
 \begin{array}{rccc}
  \pi_{\bbSigma} : & \calV(M;\bbSigma) & \rightarrow & \rmV_\calC(\bbSigma) \\
  & T & \mapsto & [M,T,0]
 \end{array} \qquad
 \begin{array}{rccc}
  \pi'_{\bbSigma} : & \calV'(M';\bbSigma) & \rightarrow & \rmV'_\calC(\bbSigma) \\
  & T' & \mapsto & [M',T',0]
 \end{array}
\]

\begin{proposition}\label{P:connection+skein}
 Let $\bbSigma = (\Sigma,P,\lambda)$ be an object of $\adCob_\calC$, let $M$ be a connected 3-dimensional cobordism from $\varnothing$ to $\Sigma$, and let $M'$ be a connected 3-dimensional cobordism from $\Sigma$ to $\varnothing$.
 \begin{enumerate}
  \item The linear map $\pi_{\bbSigma}$ is surjective, and vectors of $\calV(M;\bbSigma)$ related by a finite sequence of skein equivalences and red-to-blue operations have the same image in $\rmV_\calC(\bbSigma)$;
  \item If $\bbSigma \neq \varnothing$ is connected, then the linear map $\pi'_{\bbSigma}$ is surjective, and vectors of $\calV'(M';\bbSigma)$ related by a finite sequence of skein equivalences and red-to-blue operations have the same image in $\rmV'_\calC(\bbSigma)$.
 \end{enumerate} 
\end{proposition}

\begin{proof}
 Let us start from part (\textit{i}). First, remark that if we have a skein equivalence
 \[
  \sum_{i=1}^m \alpha_i \cdot T_i \doteq \sum_{i'=1}^{m'} \alpha'_{i'} \cdot T'_{i'}
 \]
 between vectors of $\calV(M;\bbSigma)$, then
 \[
  \sum_{i=1}^m \alpha_i \rmL'_\calC( \bbM'_{\bbSigma} \circ (M,T_i,0) ) = 
  \sum_{i'=1}^{m'} \alpha'_{i'} \rmL'_\calC( \bbM'_{\bbSigma} \circ (M,T'_{i'},0) )
 \]
 for every morphism $\bbM'_{\bbSigma} : \bbSigma \rightarrow \varnothing$ of $\adCob_\calC$. This follows directly from the very definition of $\rmL'_\calC$ in terms of the Lyubashenko-Reshetikhin-Turaev functor $F_\intL$. Therefore skein equivalent vectors of $\calV(M;\bbSigma)$ have the same image in $\rmV_\calC(\bbSigma)$. The same applies to vectors related by a red-to-blue operation.

 Next, we claim that, up to skein equivalence, we can assume every connected component of every vector in $\calV(\bbSigma)$ features a coupon of label $\epsilon_{\one}$, or one of label $\eta_{\one}$, or both. In order to show this, the idea is to use the properties of projective objects of $\calC$. Indeed, since the tensor product $\otimes$ is exact, then for every $V \in \calC$ the morphism $\epsilon_{\one} \otimes \id_V : P_{\one} \otimes V \rightarrow V$ is epic, and the morphism $\eta_{\one} \otimes \id_V : V \rightarrow P_{\one} \otimes V$ is monic. Therefore, if $V \in \Proj(\calC)$, we can always find a section $s_V: V \rightarrow P_{\one} \otimes V$, i.e. a morphism satisfying $(\epsilon_{\one} \otimes \id_V) \circ s_V = \id_V$, like in Lemma \ref{L:red_turns_blue}. Note that, thanks to the rigidity of $\calC$, projective objects are also injective, and thus similarly we can always find a retraction $r_V: P_{\one} \otimes V \rightarrow V$, i.e. a morphism satisfying $r_V \circ (\eta_{\one} \otimes \id_V) = \id_V$. This means that every time a vector of $\calV(\bbSigma)$ features a blue edge labeled by some projective object $V$, we can replace a small portion of it with one of the bichrome graphs represented in Figure \ref{projective_trick} without altering the vector in the quotient $\rmV_\calC(\bbSigma)$. We call this operation \textit{projective trick}, and we will use it in the following argument.

 \begin{figure}[hbt]
  \centering
  \includegraphics{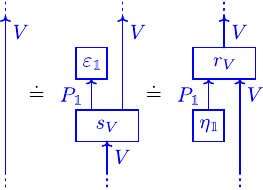}
  \caption{Projective trick along a blue edge of label $V \in \Proj(\calC)$.}
  \label{projective_trick}
 \end{figure}

 Now, to prove that $\pi_{\bbSigma}$ is surjective, we have to show that for every vector $(M_{\Sigma},T,n) : \varnothing \to \bbSigma$ there exist admissible bichrome graphs $T_1, \ldots, T_m \subset M$ and coefficients $\alpha_1, \ldots, \alpha_m \in \Bbbk$ such that
 \[
  \sum_{i = 1}^m \alpha_i \cdot [M,T_i,0] = [M_{\Sigma},T,n].
 \]
 We do this in two steps. First, we can assume that $M_{\Sigma}$ is connected: indeed every time we have distinct connected components of $M_\Sigma$ we can suppose, up to skein equivalence, one of them contains a coupon of label $\epsilon_{\one}$, while the other one contains a coupon of label $\eta_{\one}$. This uses the projective trick introduced earlier, as well as the admissibility condition for morphisms of $\adCob_\calC$. Then, thanks to Proposition \ref{P:surgery_axioms}, the 1-surgery connecting them will determine a vector of $\rmV_\calC(\bbSigma)$ which is a non-zero scalar multiple of $[M_{\Sigma},T,n]$. Second, assuming now $M_{\Sigma}$ is connected, we know there exists an $\ell$-component surgery link $L$ for $M_{\Sigma}$ inside $M$, as explained in \cite[Sec.~1.8]{BHMV95}. Then, thanks to Proposition \ref{P:surgery_axioms} with $k = 2$, there exists some signature defect $n' \in \Z$ such that 
 \[
  [M_{\Sigma},T,n] = \calD^{-\ell} \cdot [M,L \cup T,n'] = \calD^{-\ell} \delta^{n'} \cdot [M,L \cup T,0]
 \]
 where, once again, we adopt a slightly abusive notation for the pull back of the bichrome graph $T$ along the embedding of the exterior of $L$ into $M_{\Sigma}$.

 The proof of part (\textit{ii}) is almost identical, except it is easier. Indeed, the only difference is we cannot perform 1-surgeries on arbitrary disconnected cobordisms, because their connected components intersecting $\Sigma$ need not contain a projective edge. However, since $\Sigma$ is connected, there is only one connected component intersecting it, and all the other ones are closed. In particular, they only account for a scalar coefficient.
\end{proof}

\FloatBarrier

\subsection{Monoidality}\label{S:monoidality}

We use the results of Sections \ref{S:surgery_axioms} and \ref{S:connectedness} in order to prove that $\rmV_\calC$ is a TQFT.

\begin{theorem}\label{T:monoidality}
 The functor $\rmV_\calC : \adCob_\calC \to \Vect_\Bbbk$ is symmetric monoidal.
\end{theorem}

\begin{proof}
 Since both associativity and left and right unitality of the coherence data are clear, then, thanks to Proposition \ref{lax_monoidality}, we just need to prove the linear map  $\mu_{\bbSigma,\bbSigma'} : \rmV_{\calC}(\bbSigma) \otimes \rmV_{\calC}(\bbSigma') \to \rmV_{\calC}(\bbSigma \disjun \bbSigma')$ is surjective for every pair of objects $\bbSigma$, $\bbSigma'$ of $\adCob_\calC$. Let $M_\Sigma$ be a connected cobordism from $\overline{S^2}$ to $\Sigma$ and let $M_{\Sigma'}$ be a connected cobordism from $S^2$ to $\Sigma'$. Proposition \ref{P:connection+skein} implies $\rmV_\calC(\bbSigma \disjun \bbSigma')$ is generated by vectors of the form
 \[
  [(D^1 \times S^2) \cup_{S^0 \times S^2} (M_\Sigma \sqcup M_{\Sigma'}),T,0]
 \]
 for some bichrome graph $T$ inside $(D^1 \times S^2) \cup_{S^0 \times S^2} (M_{\Sigma} \sqcup M_{\Sigma'})$ from $\varnothing$ to $P \sqcup P'$. Let us choose such a $T$ and let us show that the corresponding vector of $\rmV_\calC(\bbSigma \disjun \bbSigma')$ lies in the image of $\mu_{\bbSigma,\bbSigma'}$. By definition of $\adCob_\calC$, we know $T$ admits a projective edge. Then, using Proposition \ref{P:connection+skein}, we can suppose $D^1 \times S^2$ intersects only blue edges of $T$. Furthermore, up to isotopy, we can suppose $D^1 \times S^2$ intersects a projective edge of $T$. Then, up to skein equivalence, since the tensor product of a projective object of $\calC$ with any other object of $\calC$ is projective, we can suppose $D^1 \times S^2$ is crossed by a single edge whose label $V$ is a projective object of $\calC$. Now, since there exist simple objects $V_i \in \Irr$ and morphisms $f_i \in \calC(V,P_{V_i})$ and $f'_i \in \calC(P_{V_i},V)$ for every integer $0 \leq i \leq m$ satisfying
 \[
  \id_V = \sum_{i=1}^m f'_i \circ f_i,
 \]
 we can decompose $[(D^1 \times S^2) \cup_{(S^0 \times S^2)} (M_{\Sigma} \sqcup M_{\Sigma'}),T,0]$ as
 \[
  \sum_{i=1}^m \left[ \left( (M_{\Sigma},T_{f_i},0) \disjun {}(M_{\Sigma'},T_{f'_i},0) \right) \circ (D^1 \times S^2,D^1 \times P_{(+,P_{V_i})},0) \right]
 \]
 for the blue set $P_{(+,P_{V_i})}$ inside $S^2$ given by $\{(0,0,1)\}$ with positive orientation and label $P_{V_i}$, and for some bichrome graphs $T_{f_i}$ inside $M_{\Sigma}$ from $\{ -1 \} \times \overline{P_{(+,P_{V_i})}}$ to $P$ and $T_{f'_i}$ inside $M_{\Sigma'}$ from $\{ +1 \} \times P_{(+,P_{V_i})}$ to $P'$ induced by $f_i$ and by $f'_i$ respectively, where $\overline{P_{(+,P_{V_i})}}$ denotes the blue set obtained from $P_{(+,P_{V_i})}$ by reversing its orientation. However, if $V_i \neq \one$, then
 \[
  \left[ D^1 \times S^2,D^1 \times P_{(+,P_{V_i})},0 \right] = 0.
 \]
 Indeed, this follows directly from Proposition \ref{P:connection+skein} and Lemma \ref{L:cutting} by choosing $D^1 \times \overline{S^2}$ as a 3-dimensional cobordism from $S^0 \times S^2$ to $\varnothing$. Then, if we suppose $V_i = \one$ only for every integer $1 \leq i \leq n \leq m$, we have
 \[
  [(D^1 \times S^2) \cup_{(S^0 \times S^2)} (M_{\Sigma} \sqcup M_{\Sigma'}),T,0]
  = \sum_{i=1}^n \left[ \left( (M_{\Sigma},T_{f_i},0) \disjun {}(M_{\Sigma'},T_{f'_i},0) \right) \circ \bbB_1 \right]
 \]
 for the index 1 belt tube $\bbB_1 : \varnothing \rightarrow \bbSigma_1$ introduced in Subsection \ref{S:surgery_axioms}. But Proposition \ref{P:surgery_axioms} with $k = 1$ yields the equality $[\bbB_1] = \calD \cdot [\bbA_1]$ between vectors of $\rmV_\calC(\bbSigma_1)$, where $\bbA_1 : \varnothing \rightarrow \bbSigma_1$ is the index 1 attaching tube introduced in Subsection \ref{S:surgery_axioms}. Then we have the chain of equalities 
 \begin{align*}
  &\sum_{i=1}^n \left[ \left( (M_\Sigma,T_{f_i},0) \disjun {}(M_{\Sigma'},T_{f'_i},0) \right) \circ \bbB_1 \right] \\*
  &\hspace*{\parindent} = \sum_{i=1}^n \calD \cdot \left[ \left( (M_\Sigma,T_{f_i},0) \disjun {} (M_{\Sigma'},T_{f'_i},0) \right) \circ \bbA_1 \right] \\*
  &\hspace*{\parindent} = \sum_{i=1}^n \calD \cdot \left[ \left( (M_\Sigma,T_{f_i},0) \circ \overline{\bbD^3_{\epsilon_{\one}}} \right) \disjun 
  \left( (M_{\Sigma'},T_{f'_i},0) \circ \bbD^3_{\eta_{\one}} \right) \right] \\*
  &\hspace*{\parindent} = \sum_{i=1}^n \calD \cdot \mu_{\bbSigma,\bbSigma'} 
  \left( \left[ (M_{\Sigma},T_{f_i},0) \circ \overline{\bbD^3_{\epsilon_{\one}}} \right] \otimes 
  \left[ (M_{\Sigma'},T_{f'_i},0) \circ \bbD^3_{\eta_{\one}} \right] \right)
 \end{align*}
 for the morphisms $\overline{\bbD^3_{\epsilon_{\one}}} : \varnothing \rightarrow \overline{\bbS^2_{(-,P_{\one})}}$ and $\bbD^3_{\eta_{\one}} : \varnothing \rightarrow \bbS^2_{(+,P_{\one})}$ of $\adCob_\calC$ introduced in the proof of Proposition \ref{P:surgery_axioms}. This concludes the proof of surjectivity of the linear map $\mu_{\bbSigma,\bbSigma'} : \rmV_{\calC}(\bbSigma) \otimes \rmV_{\calC}(\bbSigma') \to \rmV_{\calC}(\bbSigma \disjun \bbSigma')$.
\end{proof}

\begin{remark}\label{R:Verlinde}
 As a consequence of monoidality we get a Verlinde type formula for dualizable objects of $\adCob_\calC$: if $\bbSigma = (\Sigma,P,\lambda)$ is an object of $\adCob_\calC$ then we denote with $\overline{\bbSigma} = (\overline{\Sigma},\overline{P},\lambda)$ the object obtained from $\bbSigma$ by reversing the orientation of both $\Sigma$ and $P$. If $P$ contains a point with projective label in every connected component of $\Sigma$ then $\bbSigma^* = \overline{\bbSigma}$. Duality morphisms 
 \begin{align*}
  \lev_{\bbSigma} &: \bbSigma^* \disjun \bbSigma \to \varnothing, &
 \lcoev_{\bbSigma} &: \varnothing \to \bbSigma \disjun \bbSigma^*, \\
 \rev_{\bbSigma} &: \bbSigma \disjun \bbSigma^* \to \varnothing, &
 \rcoev_{\bbSigma} &: \varnothing \to \bbSigma^* \disjun \bbSigma
 \end{align*}
 are given by cylinders, with right evaluation and coevaluation both realized by the same decorated 3-manifold realizing the identity $\id_{\bbSigma} : \bbSigma \to \bbSigma$, although seen as different cobordisms, and with $\lev_{\bbSigma} = \rev_{\overline{\bbSigma}}$ and $\lcoev_{\bbSigma} = \rcoev_{\overline{\bbSigma}}$. Therefore, if we set $\bbS^1 \times \bbSigma := \rev_{\bbSigma} \circ \lcoev_{\bbSigma}$, we get
 \begin{align*}
  \rmL'_\calC(\bbS^1 \times \bbSigma) &= \rmV_\calC \left( \rev_{\bbSigma} \right) \circ \rmV_\calC \left( \lcoev_{\bbSigma} \right) = \rev_{\rmV_\calC(\bbSigma)} \circ \lcoev_{\rmV_\calC(\bbSigma)} \\* &= \dim_\Bbbk \left( \rmV_\calC(\bbSigma) \right).
 \end{align*}
\end{remark}

\subsection{Identification of state spaces}\label{S:identification_state_spaces}

We finish by showing that state spaces of $\rmV_\calC$ can be identified with the algebraic models defined in Section \ref{S:pairing}. Indeed, recall that we introduced for every integer $g \geq 0$ and for every $V \in \calC$ vector spaces 
\begin{gather*}
 \calX_{g,V} = \calC(P_{\one},\eend^{\otimes g} \otimes V), \quad
 \calX'_{g,V} = \calC(\coend^{\otimes g} \otimes V,\one),
\end{gather*}
as well as quotient spaces $\rmX_{g,V}$ and $\rmX'_{g,V}$ obtained from $\calX_{g,V}$ and from $\calX'_{g,V}$ by factoring respectively the right and the left radical of the bilinear pairing 
\[ 
 \brk{\cdot,\cdot}_{g,V} : \calX'_{g,V} \times \calX_{g,V} \to \Bbbk.
\]
Then, in order to exploit Proposition~\ref{P:connection+skein}, let us consider a genus $g$ Heegaard splitting $M_g \cup_{\Sigma_g} M'_g$ of $S^3$. We denote with $P_{(+,V)}$ a blue set inside $\Sigma_g$ composed of a single point with positive orientation and label $V$, we denote with $\lambda_g$ the Lagrangian subspace of $H_1(\Sigma_g;\R)$ given by the kernel of the inclusion of $\Sigma_g$ into $M_g$, and we denote with $\bbSigma_{g,V}$ the object $(\Sigma_g,P_{(+,V)},\lambda_g)$ of $\adCob_\calC$. We also fix disjoint embeddings 
\[
 \iota_g : M_g \hookrightarrow \R^3, \qquad \iota'_g : M'_g \hookrightarrow \R^3
\] 
placing $M_g$ and $M'_g$ as shown in Figure \ref{F:surgery_presentation} with respect to the surgery presentation of $S^3$ given by the $g$ red Hopf links.

\begin{figure}[htb]
  \centering
  \includegraphics{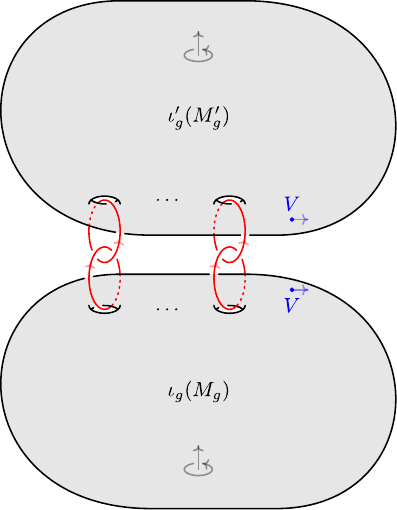}
  \caption{Surgery presentation of $M_g \cup_{\Sigma_g} M'_g$.}
  \label{F:surgery_presentation}
 \end{figure}

Let us consider now the linear map $\Psi : \calX_{g,V} \to \calV(M_g;\bbSigma_{g,V})$ sending every $f$ in $\calX_{g,V}$ to the admissible bichrome graph $\Psi(f) \subset M_g$ from $\varnothing$ to $P_{(+,V)}$ whose image under $\iota_g$ is given by
\begin{gather*}
 \raisebox{-0.5\height}{\includegraphics{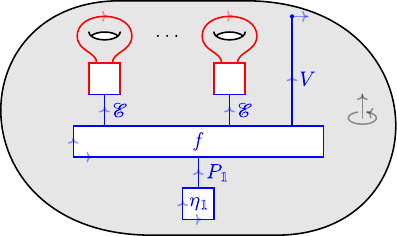}} \\
\end{gather*}

Similarly, let us consider the linear map $\Psi' : \calX'_{g,V} \to \calV'(M'_g;\bbSigma_{g,V})$ sending every $f'$ in $\calX'_{g,V}$ to the bichrome graph $\Psi'(f') \subset M'_g$ from $P_{(+,V)}$ to $\varnothing$ whose image under $\iota'_g$ is given by
\begin{gather*}
 \raisebox{-0.5\height}{\includegraphics{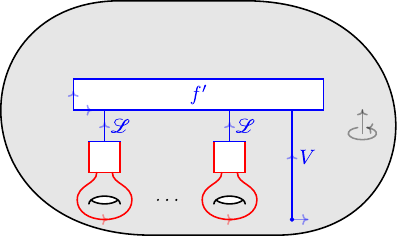}} \\
\end{gather*}

\begin{lemma}\label{L:morphism_to_skein_prime}
 The linear map $\pi'_{\bbSigma_{g,V}} \circ \Psi' : \calX'_{g,V} \to \rmV'_\calC(\bbSigma_{g,V})$ is surjective.
\end{lemma}

\begin{proof}
 Thanks to Proposition \ref{P:connection+skein}, we only need to show that for every bichrome graph $T' \in \calV'(M'_g;\bbSigma_{g,V})$ there exists a morphism $f' \in \calX'_{g,V}$ satisfying 
 \[
  [M'_g,T',0] = [M'_g,\Psi'(f'),0].
 \]
 Note that if $K \subset M'_g \smallsetminus T'$ is a red knot and $K' \subset M'_g \smallsetminus (T' \cup K)$ is a red meridian of $K$ then, since surgery along $K'$ reverses the effect of performing surgery along $K$, Proposition \ref{P:surgery_axioms} with $k = 2$ yields 
 \[
  [M'_g,T',0] = \calD^{-2} [M'_g,T' \cup K \cup K',0].
 \]
 Therefore we can choose, for every integer $1 \leq i \leq g$, a red knot $K_i \subset M'_g \smallsetminus T'$ which runs along the core of the $i$-th index 1 handle of $M'_g$. Then, up to isotopy, we can suppose every blue or red edge of $T'$ crossing the handle appears as in the left-hand side of Figure \ref{F:slide_trick}, and we can slide it off the handle and into the meridian $K'_i$ as shown in the right-hand side. 

 \begin{figure}[h]
  \centering
  \includegraphics{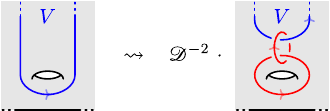}
  \caption{Slide trick.}
  \label{F:slide_trick}
 \end{figure} 
 
 The resulting bichrome graph $T''$ can be written as the plat closure of a $g$-bottom graph $\tilde{T}'' \in (g)\calR_\intL((V,+),\varnothing)$ embedded in $M'_g$ as shown in Figure \ref{F:bottom_graph}. Now let $\tilde{T}''' \in (n+g)\calR_\intL((V,+),\varnothing)$ be a complete $n+g$-bottom graph presentation of $T''$ whose partial plat closure along the $n$ left-most pairs of boundary vertices coincides with $\tilde{T}''$. Then, following the procedure for the definition of the functor $F_\intL$, as in Subsection \ref{SS:LRT_functor}, $\tilde{T}'''$ induces a morphism 
 \[
  f_\calC(\eta_{\tilde{T}'''}) \in \calC(\coend^{\otimes n+g} \otimes V,\one).
 \]
 The result is obtained by setting,
 \[
  f' := \calD^{-2g} \cdot f_\calC \left( \eta_{\tilde{T}'''} \right) \circ (\intL^{\otimes n} \otimes \id_{\coend^{\otimes g} \otimes V})
 \] 
 since by construction $[M'_g,T',0] = [M'_g,\Psi'(f'),0]$.

 \begin{figure}[b]
  \centering
  \includegraphics{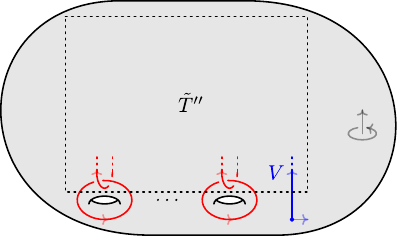}
  \caption{Presentation of the bichrome graph $T''$ as the plat closure of the $g$-bottom graph $\tilde{T}''$.}
  \label{F:bottom_graph}
 \end{figure}

\end{proof}

\begin{remark}\label{R:dimension}
 Lemma \ref{L:morphism_to_skein_prime} yields
 \[
  \dim_\Bbbk \rmX_{g,V} = \dim_\Bbbk \rmX'_{g,V} = \dim_\Bbbk \calX'_{g,V} \geq \dim_\Bbbk \rmV'_\calC(\bbSigma_{g,V}) = \dim_\Bbbk \rmV_\calC(\bbSigma_{g,V}).
 \]
\end{remark}

\FloatBarrier

\begin{lemma}\label{L:morphism_to_skein}
 There exists an injective linear map $\Phi : \rmX_{g,V} \to \rmV_\calC(\bbSigma_{g,V})$ yielding the commutative diagram
 \begin{center}
  \begin{tikzpicture}[descr/.style={fill=white}]
   \node (P0) at (45:{sqrt(2)}) {$\calV(M_g;\bbSigma_{g,V})$};
   \node (P1) at (45+90:{sqrt(2)}) {$\calX_{g,V}$};
   \node (P2) at (45+2*90:{sqrt(2)}) {$\rmX_{g,V}$};
   \node (P3) at (45+3*90:{sqrt(2)}) {$\rmV(\bbSigma_{g,V})$};
   \draw
   (P1) edge[->] node[above] {$\Psi$} (P0)
   (P1) edge[->] node[left] {$[ {} \cdot {} ]$} (P2)
   (P2) edge[->] node[below] {$\Phi$} (P3)
   (P0) edge[->] node[right] {$\pi_{\bbSigma_{g,V}}$} (P3);
  \end{tikzpicture}
 \end{center}
\end{lemma}

\begin{proof}
 \begin{figure}[b]
  \centering
  \includegraphics{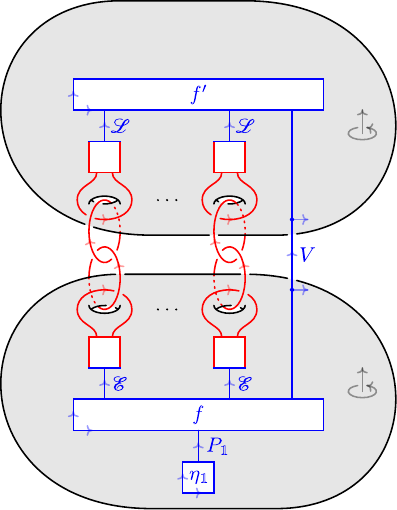}
  \caption{The admissible bichrome graph $\tilde{T}_{f',f}$.}
  \label{F:Heegaard_splitting}
 \end{figure}
 We need to show that $f \in \ker [ {} \cdot {} ] \subset \calX_{g,V}$ if and only if
 \[
  [M_g,\Psi(f),0] = 0 \in \rmV(\bbSigma_{g,V}).
 \]
 Thanks to Lemma \ref{L:morphism_to_skein_prime}, this happens if and only if
 \[
  \rmL'_\calC((M'_g,\Psi'(f'),0) \circ (M_g,\Psi(f),0)) = 0
 \]
 for every $f' \in \calX'_{g,V}$. This invariant can be computed using the admissible bichrome graph $\tilde{T}_{f',f}$ of Figure \ref{F:Heegaard_splitting}. Note that we can disentangle all red Hopf links from the rest of the bichrome graph by using them to slide all bottom red cycles upwards and all top red cycles downwards. Once this has been done, we can remove all red Hopf links by appealing to Proposition~\ref{P:surgery_axioms}, since they induce trivial surgeries. This yields the admissible bichrome graph $T_{f',f}$ of Figure \ref{F:comparison_of_pairings}. Then Lemma \ref{L:comparison_of_pairings} gives 
 \begin{align*}
  \rmL'_\calC((M'_g,\Psi'(f'),0) \circ (M_g,\Psi(f),0)) &= \calD^{-1-2g} F'_\intL(\tilde{T}_{f',f}) \\*
  &= \calD^{-1} F'_\intL(T_{f',f}) \\*
  &= \calD^{-1+2g} \langle f',f \rangle_{g,V}. \qedhere
 \end{align*}
\end{proof}

\begin{proposition}\label{P:state-space-iso}
 The linear map $\Phi : \rmX_{g,V} \to \rmV_\calC(\bbSigma_{g,V})$ is an isomorphism.
\end{proposition}

\begin{proof}
 The claim is a direct consequence of Remark \ref{R:dimension} and of Lemma \ref{L:morphism_to_skein}.
\end{proof}

\FloatBarrier

\addtocontents{toc}{\protect\setcounter{tocdepth}{0}}

\subsection*{Acknowledgments}

The present work was initiated at the meeting \textit{Non-Sem\-i\-sim\-ple TFT and Logarithmic CFT} (Hamburg 2018), which was funded by the Research Training Group 1670 of the DFG. The authors would like to thank the organizers of the conferences \textit{TQFT and Categorification} (Cargèse 2018) and \textit{Geometric and Categorical Aspects of CFTs} (Oaxaca 2018), where the foundations for this paper were further developed. The authors are also grateful to the referee for many helpful comments. MD is supported by JSPS and partially by KAKENHI Grant-in-Aid for JSPS Fellows 19F19765. AMG is supported by CNRS, and thanks Humboldt Foundation and ANR grant JCJC ANR-18-CE40-0001 for a partial financial support. AMG is also grateful to Utah State University for its kind hospitality during June 2019. NG was partially supported by the NSF grant DMS-1452093. IR is partially supported by the Deutsche Forschungsgemeinschaft via the Research Training Group 1670 and the Cluster of Excellence EXC 2121.

\addtocontents{toc}{\protect\setcounter{tocdepth}{2}}

\appendix

\section{Proofs}\label{A:proofs}

In this appendix we collect some technical proofs of results which were announced in earlier sections. All the arguments follow a recurring pattern, which is why we gather them all here.

\addtocontents{toc}{\protect\setcounter{tocdepth}{0}}

\subsection{Proof of results from Section \ref{SS:LRT_functor}}\label{S:proof_LRT_functor}

\begin{proof}[Proof of Proposition \ref{P:LRT_functor}]
 
 \begin{figure}[b]
  \centering
  \includegraphics{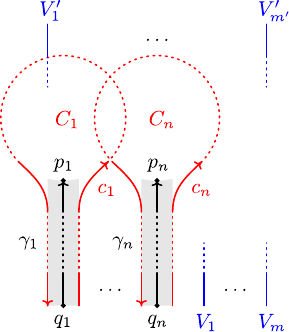}
  \caption{Diagram representing $\tilde{T}$.}
  \label{F:LRT_functor}
 \end{figure}
  
 Before showing that $F_\intL : \calR_\intL \to \calC$ is a well-defined monoidal functor, let us start by showing that if $T : (\underline{\epsilon},\underline{V}) \to (\underline{\epsilon'},\underline{V'})$ is a morphism of $\calR_\intL$ featuring exactly $n$ cycles $C_1, \ldots, C_n$, then there exists a complete $n$-bottom graph presentation $\tilde{T}$ of $T$. In order to define it, let us choose basepoints $p_k$ along some red edge $c_k \in C_k$ and $q_k$ in the interior of the arc $e_k \subset \R \times \{ (0,0) \} \subset \R^2 \times I$ joining the $2k-1$th and the $2k$th unlabeled red vertices of $(n)(\underline{\epsilon},\underline{V})$ for every integer $1 \leq k \leq n$. Then we consider pairwise disjoint embeddings $\iota_1, \ldots, \iota_n$ of $D^1 \times D^1$ into $\R^2 \times I$ with $\iota_k((D^1 \smallsetminus \partial D^1) \times D^1)$ contained in the complement of $T$, with $(\iota_k(\{ -1 \} \times D^1),\iota_k(-1,0)) \subset (c_k,p_k)$, and with $(\iota_k(\{ 1 \} \times \overline{D^1}),\iota_k(1,0)) = (e_k,q_k)$ for every integer $1 \leq k \leq n$. We denote with $\gamma_k$ the framed path from $q_k$ to $p_k$ determined by $\iota_k(D^1 \times \{ 0 \})$ with framing orthogonal to the image of $\iota_k$, and we let $\tilde{T}$ be the $n$-bottom graph obtained from $T$ by replacing $\iota_k(\{ -1 \} \times D^1)$ with $\iota_k(D^1 \times \partial D^1)$ for every integer $1 \leq k \leq n$. Up to isotopy, this morphism can be represented by a diagram like the one depicted in Figure \ref{F:LRT_functor}. Then, by construction,
 \[
  \pc (\tilde{T}) = T.
 \] 
 Now, in order to prove $F_\intL$ is a well-defined monoidal functor, we will first show its definition is independent of the choice of framed paths $\gamma_1, \ldots, \gamma_n$, of basepoints $p_1, \ldots, p_n$, and of the ordering of cycles $C_1, \ldots, C_n$, and then we will show functoriality and monoidality.
 
 \begin{figure}[t]
  \centering
  \includegraphics{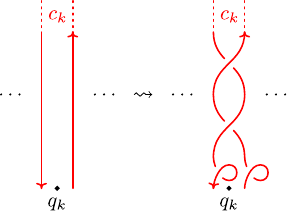}
  \caption{Kink insertion.}
  \label{F:kink_insertion}
 \end{figure}
 \begin{figure}[b]
  \centering
  \includegraphics{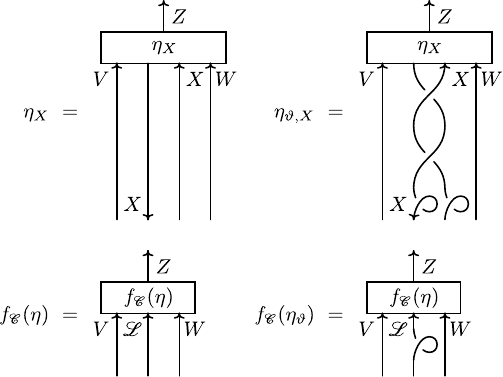}
  \caption{Component $\eta_{\theta,X}$ and induced morphism $f_\calC(\eta_\theta)$.}
  \label{F:independence_of_framings}
 \end{figure}
 
 \textit{Independence of framings.} First, we claim $F_\intL$ is independent of the choice of framings. In order to prove this, it is sufficient to show that we can insert a positive kink on any framed path $\gamma_k$ for every integer $1 \leq k \leq n$. Up to isotopy, a kink insertion along $\gamma_k$ is represented in Figure \ref{F:kink_insertion}. Note that, since isotopies do not affect $F_\calC$, they do not affect $f_\calC$ in Equation \eqref{E:f_calC} either. Then the claim is a consequence of the following general argument: for all objects $V,W \in \calC$ let us denote with $H_{V,W} : \calC^\op \times \calC \to \calC$ the functor sending every object $(X,Y)$ of $\calC^\op \times \calC$ to the object $V \otimes X^* \otimes Y \otimes W$ of $\calC$. The universal property defining $\coend$ implies the object $V \otimes \coend \otimes W$ equipped with the morphisms $\id_V \otimes i_X \otimes \id_W$ for every $X \in \calC$ is the coend for the functor $H_{V,W}$. Then every object $Z \in \calC$ and every dinatural transformation $\eta : H_{V,W} \din Z$ induce a unique morphism $f_\calC(\eta) \in \calC(V \otimes \coend \otimes W,Z)$ satisfying
 \[
  f_\calC(\eta) \circ (\id_V \otimes i_X \otimes \id_W) = \eta_X.
 \]
 Now let us consider the dinatural transformation $\eta_\theta : H_{V,W} \din Z$ whose component $\eta_{\theta,X}$ is represented in the top right part of Figure \ref{F:independence_of_framings} for every object $X \in \calC$. The morphism $f_\calC(\eta_\theta) \in \calC(V \otimes \coend \otimes W,Z)$ induced by $\eta_\theta$ is represented in the bottom right part of Figure \ref{F:independence_of_framings}. Indeed, this follows immediately from the naturality of the twist $\theta$ and from Equation \eqref{eq:theta-braiding-prop}, which allow us to check that the morphism of Figure \ref{F:independence_of_framings} safisfies the defining equation for $f_\calC(\eta_\theta)$, which is
 \[
  f_\calC(\eta_\theta) \circ ( \id_V \otimes i_X \otimes \id_W ) = \eta_{\theta,X}.
 \]
 Then, since $\intL$ is a morphism from $\one$ to $\coend$, this means
 \begin{align*}
  f_\calC(\eta_{\theta}) \circ (\id_V \otimes \intL \otimes \id_W) &= f_\calC(\eta) \circ ( \id_V \otimes (\theta \circ \intL) \otimes \id_W ) \\*
  &= f_\calC(\eta) \circ ( \id_V \otimes \intL \otimes \id_W ).
 \end{align*}
 Therefore our claim is a direct application of the equality we just established for a fixed choice of the variables $X_1, \ldots, X_{k-1},X_{k+1}, \ldots, X_n \in \calC$, with 
 \begin{align*}
  V &= X_1^* \otimes X_1 \otimes \ldots \otimes X_{k-1}^* \otimes X_{k-1}, \\*
  W &= X_{k+1}^* \otimes X_{k+1} \otimes \ldots \otimes X_n^* \otimes X_n \otimes F_\calC(\underline{\epsilon},\underline{V}), \\*
  Z &= F_\calC(\underline{\epsilon'},\underline{V}'), \\*
  \eta &= F_\calC(\tilde{T}_{(X_1, \ldots, X_{k-1},X_{k+1}, \ldots, X_n)}),
 \end{align*}
 where for every $X_k \in \calC$ the component $F_\calC(\tilde{T}_{(X_1, \ldots, X_{k-1},X_{k+1}, \ldots, X_n)})_{X_k}$ of the dinatural transformation $F_\calC(\tilde{T}_{(X_1, \ldots, X_{k-1},X_{k+1}, \ldots, X_n)})$ is given by $F_\calC(\tilde{T}_{(X_1, \ldots, X_n)})$, with $\tilde{T}_{(X_1, \ldots, X_n)}$ defined as in the construction of $F_\intL(T)$.
 
 \begin{figure}[b]
  \centering
  \includegraphics{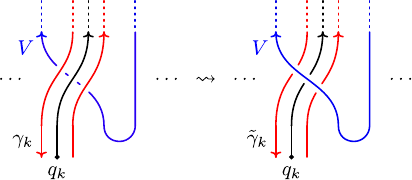}
  \caption{Crossing exchange.}
  \label{F:crossing_exchange}
 \end{figure}
 \begin{figure}[t]
  \centering
  \includegraphics{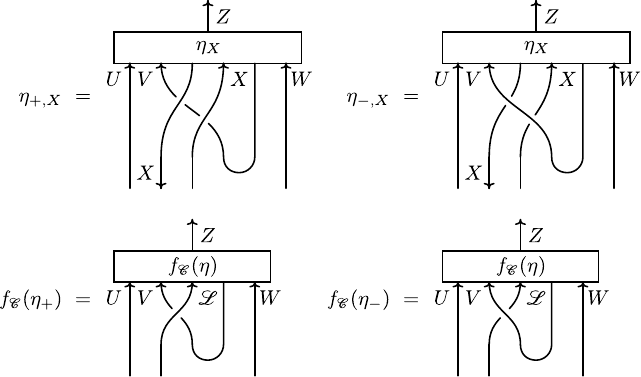}
  \caption{Components $\eta_{+,X}$ and $\eta_{-,X}$ and induced morphisms $f_\calC(\eta_+)$ and $f_\calC(\eta_-)$.}
  \label{F:independence_of_paths}
 \end{figure}
 
 \textit{Independence of paths.} Next, we claim $F_\intL$ is independent of the choice of framed paths $\gamma_k$ connecting $q_k$ to $p_k$ for every integer $1 \leqslant k \leqslant n$. In order to prove this, it is sufficient to show that every overcrossing of $\gamma_k$ with the rest of $\tilde{T}$ can be exchanged for an undercrossing. Up to isotoping the desired double point to the bottom of the diagram, an exchange of a crossing of $\gamma_k$ with a blue edge is represented in Figure \ref{F:crossing_exchange}, and the same picture applies to an exchange of a crossing of $\gamma_k$ with a red edge. Then, just like we argued before, the claim is a consequence of the following general argument: for all objects $U,V,W \in \calC$ let us denote with $H_{U,V,W} : \calC^\op \times \calC \to \calC$ the functor sending every object $(X,Y)$ of $\calC^\op \times \calC$ to the object ${U \otimes V \otimes X^* \otimes Y \otimes V^* \otimes W}$ of $\calC$. Then, for every object $Z \in \calC$ and every dinatural transformation $\eta : H_{U,V,W} \din Z$, let us consider the dinatural transformations $\eta_+,\eta_- : H_{U,\one,W} \din Z$ whose components $\eta_{+,X}$ and $\eta_{-,X}$ are represented in the top part of Figure \ref{F:independence_of_paths} for every object $X \in \calC$. The morphisms $f_\calC(\eta_+),f_\calC(\eta_-) \in \calC(U \otimes \coend \otimes W,Z)$ induced by $\eta_+$ and $\eta_-$ are represented in the bottom part of Figure \ref{F:independence_of_paths}. Then, since $\intL$ is a morphism from $\one$ to $\coend$, this means
 \begin{align*}
  &f_\calC(\eta_+) \circ (\id_U \otimes \intL \otimes \id_W) \\*
  &\hspace{\parindent}= f_\calC(\eta) \circ ( \id_U \otimes ((c_{\coend,V} \otimes \id_{V^*}) \circ (\intL \otimes \lcoev_V)) \otimes \id_W ) \\*
  &\hspace{\parindent}= f_\calC(\eta) \circ ( \id_U \otimes ((c^{-1}_{V,\coend} \otimes \id_{V^*}) \circ (\intL \otimes \lcoev_V)) \otimes \id_W ) \\*
  &\hspace{\parindent}= f_\calC(\eta_-) \circ (\id_U \otimes \intL \otimes \id_W).
 \end{align*}

 \FloatBarrier

 \begin{figure}[b]
  \centering
  \includegraphics{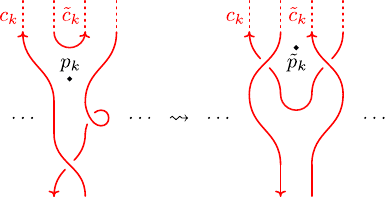}
  \caption{Basepoint change.}
  \label{F:basepoint_change}
 \end{figure}
 \begin{figure}[t]
  \centering
  \includegraphics{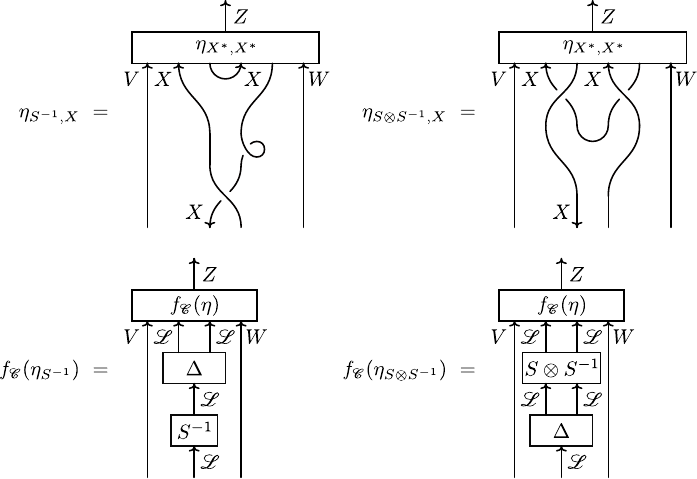}
  \caption{Components $\eta_{\antipL^{-1},X}$ and $\eta_{\antipL \otimes \antipL^{-1},X}$ and induced morphisms $f_\calC(\eta_{\antipL^{-1}})$ and $f_\calC(\eta_{\antipL \otimes \antipL^{-1}})$.}
  \label{F:independence_of_basepoints}
 \end{figure} 
 
 \textit{Independence of basepoints.} Next, we claim $F_\intL$ is independent of the choice of basepoints $p_k \in c_k \in C_k$ for every integer $1 \leqslant k \leqslant n$. Indeed, let $\tilde{p}_k \in \tilde{c}_k \in C_k$ be another possible choice. Up to isotoping portions of $c_k$ and $\tilde{c}_k$ containing $p_k$ and $\tilde{p}_k$ respectively to the bottom of the diagram, making sure the one containing $\tilde{p}_k$ is nested inside the one containing $p_k$ with opposite orientation, this operation is represented in Figure \ref{F:basepoint_change}. Then, once again, the claim is a consequence of the following general argument: for all objects $V,W \in \calC$ let us denote with ${H_{V,W} : (\calC^\op \times \calC)^{\times 2} \to \calC}$ the functor sending every object $(X_1,Y_2,X_1,X_2)$ of $(\calC^\op \times \calC)^{\times 2}$ to the object $V \otimes X_1^* \otimes Y_1 \otimes X_2^* \otimes Y_2 \otimes W$ of $\calC$. Then, for every object $Z \in \calC$ and every $2$-dinatural transformation $\eta : H_{V,W} \din Z$, let us consider the dinatural transformations $\eta_{\antipL^{-1}},\eta_{\antipL \otimes \antipL^{-1}} : H_{V,W} \din Z$ whose components $\eta_{\antipL^{-1},X}$ and $\eta_{\antipL \otimes \antipL^{-1},X}$ are represented in the top part of Figure \ref{F:independence_of_basepoints} for every object $X \in \calC$. The morphisms $f_\calC(\eta_{\antipL^{-1}}),f_\calC(\eta_{\antipL \otimes \antipL^{-1}}) \in \calC(V \otimes \coend \otimes W,Z)$ induced by $\eta_{\antipL^{-1}}$ and $\eta_{\antipL \otimes \antipL^{-1}}$ are represented in the bottom part of Figure \ref{F:independence_of_basepoints}, as follows from the definition of the coproduct $\copL$ and of the antipode $\antipL$ in Figure \eqref{E:antipode}. Then, thanks to Equation \eqref{eq:brHopf-identity}, this gives
 \begin{align*}
  &f_\calC(\eta_{\antipL^{-1}}) \circ (\id_V \otimes \intL \otimes \id_W) \\*
  &\hspace{\parindent}= f_\calC(\eta) \circ ( \id_V \otimes (\copL \circ \antipL^{-1} \circ \intL) \otimes \id_W ) \\*
  &\hspace{\parindent}= f_\calC(\eta) \circ ( \id_V \otimes ((\antipL \otimes \antipL^{-1}) \circ \copL \circ \intL) \otimes \id_W ) \\*
  &\hspace{\parindent}= f_\calC(\eta_{\antipL \otimes \antipL^{-1}}) \circ (\id_V \otimes \intL \otimes \id_W).
 \end{align*}

 \FloatBarrier

 \begin{figure}[b]
  \centering
  \includegraphics{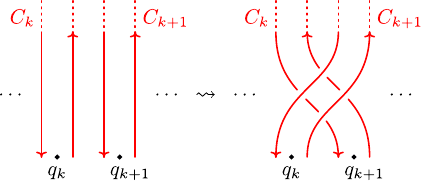}
  \caption{Cycle transposition.}
  \label{F:cycle_transposition}
 \end{figure}
 \begin{figure}[t]
  \centering
  \includegraphics{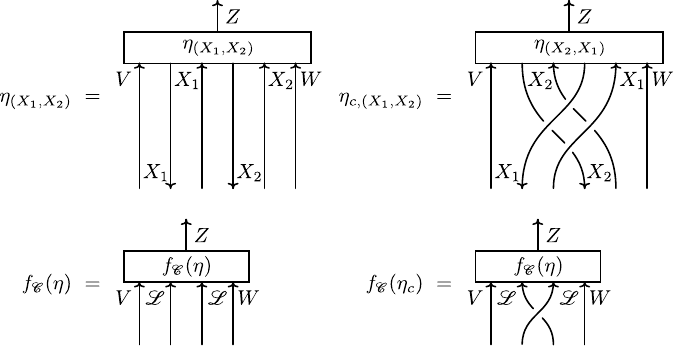}
  \caption{Component $\eta_{c,(X_1,X_2)}$ and induced morphism $f_\calC(\eta_c)$.}
  \label{F:independence_of_order}
 \end{figure}

 \textit{Independence of order.} Now, we claim $F_\intL$ is independent of the choice of the ordering of cycles. In order to prove this, it is sufficient to show that cycles $C_k$ and $C_{k+1}$ can be transposed for every integer $1 \leqslant k < n$. This operation is represented in Figure \ref{F:cycle_transposition}. Then, once again, the claim is a consequence of the following general argument: for all objects $V,W \in \calC$ let us denote with $H_{V,W} : (\calC^\op \times \calC)^{\times 2} \to \calC$ the functor sending every object $(X_1,Y_1,X_2,Y_2)$ of $(\calC^\op \times \calC)^{\times 2}$ to the object $V \otimes X_1^* \otimes Y_1 \otimes X_2^* \otimes Y_2 \otimes W$ of $\calC$. Then, for every object $Z \in \calC$ and every $2$-dinatural transformation $\eta : H_{V,W} \din Z$, let us consider the $2$-dinatural transformation ${\eta_c : H_{V,W} \din Z}$ whose component $\eta_{c,(X_1,X_2)}$ is represented in the top-right part of Figure \ref{F:independence_of_order} for every object $(X_1,X_2) \in \calC^{\times 2}$. The morphism $f_\calC(\eta_c) \in \calC(V \otimes \coend \otimes \coend \otimes W,Z)$ induced by $\eta_c$ is represented in the bottom-right part of Figure \ref{F:independence_of_order}. Then, since $\intL$ is a morphism from $\one$ to $\coend$, this means
 \begin{align*}
  f_\calC(\eta_c) \circ (\id_V \otimes \intL \otimes \intL \otimes \id_W) &= f_\calC(\eta) \circ ( \id_V \otimes (c_{\coend,\coend} \circ (\intL \otimes \intL)) \otimes \id_W ) \\*
  &= f_\calC(\eta) \circ ( \id_V \otimes \intL \otimes \intL \otimes \id_W ).
 \end{align*}

 \FloatBarrier

 \begin{figure}[b]
  \centering
  \includegraphics{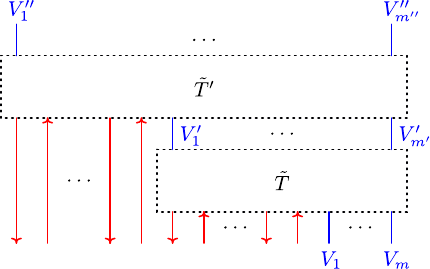}
  \caption{Complete $n+n'$-bottom graph presentation of $T' \circ T$.}
  \label{F:functoriality}
 \end{figure}

 \textit{Functoriality.} If $(\underline{\epsilon},\underline{V})$ is an object of $\calR_\intL$, then we clearly have
 \[
  F_\intL(\id_{(\underline{\epsilon},\underline{V})}) = \id_{F_\intL(\underline{\epsilon},\underline{V})}.
 \]
 If $T : (\underline{\epsilon},\underline{V}) \to (\underline{\epsilon'},\underline{V'})$ and $T' : (\underline{\epsilon'},\underline{V'}) \to (\underline{\epsilon''},\underline{V''})$ are morphisms of $\calR_\intL$, if $\tilde{T}$ is a complete $n$-bottom graph presentation of $T$, and if $\tilde{T}'$ is a complete $n'$-bottom graph presentation of $T'$, then a complete $n+n'$-bottom graph presentation of $T' \circ T$ is represented in Figure \ref{F:functoriality}. Therefore
 \[
  F_\intL(T' \circ T) = F_\intL(T') \circ F_\intL(T).
 \]

 \textit{Monoidality.} If $(\underline{\epsilon},\underline{V})$ and $(\underline{\epsilon'},\underline{V'})$ are objects of $\calR_\intL$, then we clearly have
 \[
  F_\intL((\underline{\epsilon},\underline{V}) \otimes (\underline{\epsilon'},\underline{V'})) = F_\intL(\underline{\epsilon},\underline{V}) \otimes F_\intL(\underline{\epsilon'},\underline{V'}).
 \]
 If $T : (\underline{\epsilon},\underline{V}) \to (\underline{\epsilon'},\underline{V'})$ and $T' : (\underline{\epsilon''},\underline{V''}) \to (\underline{\epsilon'''},\underline{V'''})$ are morphisms of $\calR_\intL$, if $\tilde{T}$ is a complete $n$-bottom graph presentation of $T$, and if $\tilde{T}'$ is a complete $n'$-bottom graph presentation of $T'$, then a complete $n+n'$-bottom graph presentation of $T \otimes T'$ is represented in Figure \ref{F:monoidality}. Therefore
 \[
  F_\intL(T \otimes T') = F_\intL(T) \otimes F_\intL(T'). \qedhere
 \]

 \begin{figure}[t]
  \centering
  \includegraphics{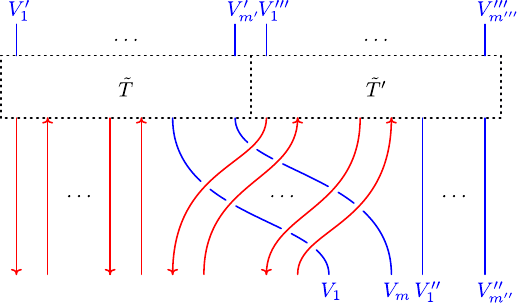}
  \caption{Complete $n+n'$-bottom graph presentation of $T \otimes T'$.}
  \label{F:monoidality}
 \end{figure}

 \FloatBarrier

\end{proof}

\subsection{Proof of results from Section \ref{SS:3-manifold_invariants}}\label{S:proof_orient_slide}

\begin{proof}[Proof of Proposition \ref{P:orient}]

 \begin{figure}[b]
  \centering
  \includegraphics{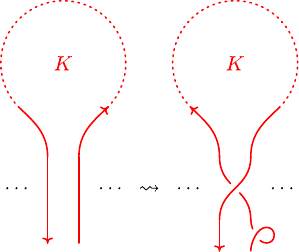}
  \caption{Orientation reversal.}
  \label{F:orientation_reversal}
 \end{figure} 
 \begin{figure}[t]
  \centering
  \includegraphics{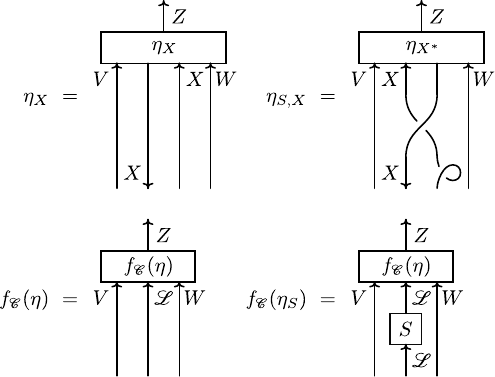}
  \caption{Component $\eta_{\antipL,X}$ and induced morphism $f_\calC(\eta_\antipL)$.}
  \label{F:independence_of_orientation}
 \end{figure}

 The proof follows from the analogous result for the original Lyubashenko invariant, see \cite[Prop.~5.2.1~\&~Fig.~10]{L94}. Indeed, let $T' \in (n)\calR_\intL((\underline{\epsilon},\underline{V}),(\underline{\epsilon'},\underline{V'}))$ be an $n$-bottom graph presentation of $T$, with $K$ appearing in $k$th position. The operation of reversing the orientation of $K$ affects $T'$ as shown in Figure \ref{F:orientation_reversal}. Now, as usual, the statement is a consequence of the following general argument: for all objects $V,W \in \calC$ let us denote with $H_{V,W} : \calC^\op \times \calC \to \calC$ the functor sending every object $(X,Y)$ of $\calC^\op \times \calC$ to the object $V \otimes X^* \otimes Y \otimes W$ of $\calC$. Then, for every object $Z \in \calC$ and every dinatural transformation $\eta : H_{V,W} \din Z$, let us consider the dinatural transformation $\eta_\antipL : H_{V,W} \din Z$ whose component $\eta_{\antipL,X}$ is represented in the top-right part of Figure \ref{F:independence_of_orientation} for every object $X \in \calC$. The morphism $f_\calC(\eta_\antipL) \in \calC(V \otimes \coend \otimes W,Z)$ induced by $\eta_\antipL$ is represented in the bottom-right part of Figure \ref{F:independence_of_orientation}, as follows from the definition of the antipode $\antipL$ in Equation~\eqref{E:antipode}. Then, thanks to part (\textit{ii}) of Lemma \ref{L:identities}, this means
 \begin{align*}
  f_\calC(\eta_\antipL) \circ ( \id_V \otimes \intL \otimes \id_W ) &= f_\calC(\eta) \circ (\id_V \otimes (\antipL \circ \intL) \otimes \id_W) \\*
  &= f_\calC(\eta) \circ (\id_V \otimes \intL \otimes \id_W).
 \end{align*}
 The statement for $F_\intL$ immediately implies the one for $F'_\intL$, since the latter is obtained from the former by applying the trace $\rmt$.
\end{proof}

 \FloatBarrier

\begin{proof}[Proof of Proposition \ref{P:slide}]

 \begin{figure}[b]
  \centering
  \includegraphics{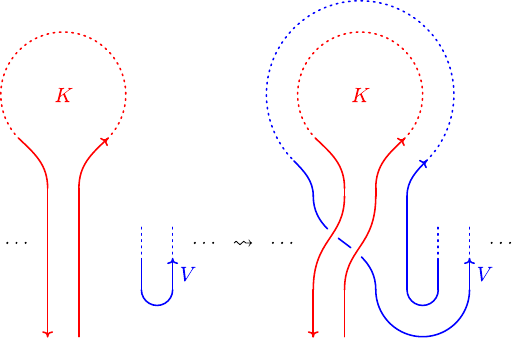}
  \caption{Edge slide.}
  \label{F:edge_slide}
 \end{figure}
 \begin{figure}[t]
  \centering
  \includegraphics{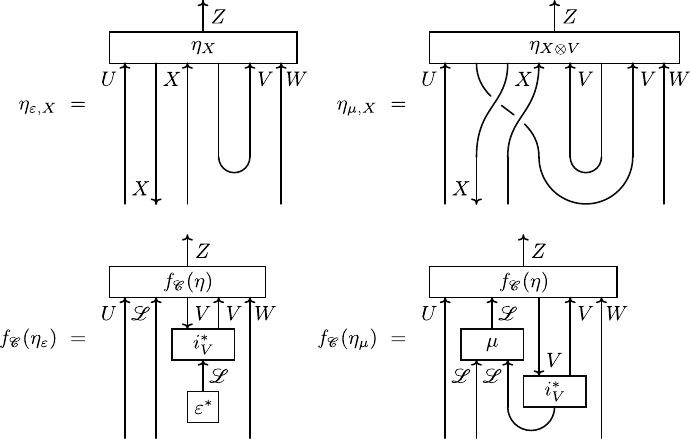}
  \caption{Components $\eta_{\epsilon,X}$ and $\eta_{\mu,X}$ and induced morphisms $f_\calC(\eta_\epsilon)$ and $f_\calC(\eta_\mu)$.}
  \label{F:invariance_under_slide}
 \end{figure}

 The proof follows from the analogous result for the original Lyubashenko invariant, see \cite[Thm.~5.2.2~\&~Fig.~11--13]{L94}. Indeed, let $T' \in (n)\calR_\intL((\underline{\epsilon},\underline{V}),(\underline{\epsilon'},\underline{V'}))$ be an $n$-bottom graph presentation of $T$, with $K$ appearing in $k$th position. The operation of sliding a blue edge over $K$ affects $T'$ as shown in Figure \ref{F:edge_slide}, and the same picture applies for the slide of a red edge. Now, as usual, the statement is a consequence of the following general argument: for all objects $U,V,W \in \calC$ let us denote with $H_{U,V,W} : \calC^\op \times \calC \to \calC$ the functor sending every object $(X,Y)$ of $\calC^\op \times \calC$ to the object $U \otimes X^* \otimes Y \otimes V^* \otimes V \otimes W$ of $\calC$. Then, for every object $Z \in \calC$ and every dinatural transformation $\eta : H_{U,V,W} \din Z$, let us consider the dinatural transformations $\eta_\counitL,\eta_\multL : H_{U,\one,W} \din Z$ whose components $\eta_{\counitL,X}$ and $\eta_{\multL,X}$ are represented in the top part of Figure \ref{F:invariance_under_slide} for every object $X \in \calC$. The morphisms $f_\calC(\eta_\counitL),f_\calC(\eta_\multL) \in \calC(U \otimes \coend \otimes W,Z)$ induced by $\eta_\counitL$ and $\eta_\multL$ are represented in the bottom part of Figure \ref{F:invariance_under_slide}, as follows from the definition of the product $\multL$ and of the counit $\counitL$ in Equations~\eqref{E:algebra_structure} and \eqref{E:coalgebra_structure}. Then, thanks to Equation \eqref{eq:right-integral-def}, this means
 \begin{align*}
  &f_\calC(\eta_\multL) \circ (\id_U \otimes \intL \otimes \id_W) \\*
  &\hspace{\parindent}= f_\calC(\eta) \circ ( \id_U \otimes (((\multL \circ (\intL \otimes \id_\coend)) \otimes i_V^*) \circ \lcoev_\coend) \otimes \id_W ) \\*
  &\hspace{\parindent}= f_\calC(\eta) \circ ( \id_U \otimes (((\intL \circ \counitL) \otimes i_V^*) \circ \lcoev_\coend) \otimes \id_W ) \\*
  &\hspace{\parindent}= f_\calC(\eta_\counitL) \circ ( \id_U \otimes \intL \otimes \id_W ).
 \end{align*}
 The statement for $F_\intL$ immediately implies the one for $F'_\intL$, since the latter is obtained from the former by applying the trace $\rmt$.
\end{proof}

\subsection{Proof of results from Section \ref{S:skein_equivalence}}\label{S:proof_red_turns_blue}

\begin{proof}[Proof of Lemma \ref{L:red_turns_blue}]
 Let us start by remarking that both $i_G \in \calC(G^* \otimes G,\coend)$ and $\epsilon_{\one} \otimes \id_V \in \calC(P_{\one} \otimes V,V)$ are epimorphisms. Indeed, the claim for $i_G$ follows from the fact that $G$ is a projective generator, see \cite[Cor.~5.1.8]{KL01}, while the one for $\epsilon_{\one} \otimes \id_V$ follows from the fact that the tensor product $\otimes$ is exact and that $\epsilon_{\one}$ is an epimorphism. Then, the existence of $f_\intL \in \calC(P_{\one},G^* \otimes G)$ and of $s_V \in \calC(V,P_{\one} \otimes V)$ follows from the fact that $P_{\one}$ and $V$ are projective. Now, invariance under the red-to-blue operation defined by Figure \ref{F:red_turns_blue} is a consequence of the following general argument: for all objects $U,W \in \calC$ let us denote with ${H_{U,W} : \calC^\op \times \calC \to \calC}$ the functor sending every object $(X,Y)$ of $\calC^\op \times \calC$ to the object ${U \otimes X^* \otimes Y \otimes V \otimes V^* \otimes W}$ of $\calC$. Then, for every object $Z \in \calC$ and every dinatural transformation $\eta : H_{U,W} \din Z$, the morphism ${f_\calC(\eta) \in \calC(U \otimes \coend \otimes V \otimes V^* \otimes W,Z)}$ induced by $\eta$ satisfies 
 \begin{align*}
  &f_\calC(\eta) \circ ( \id_U \otimes \intL \otimes \lcoev_V \otimes \id_W ) \\*
  &\hspace{\parindent}= f_\calC(\eta) \circ \left( \id_U \otimes \intL \otimes \left( (((\epsilon_{\one} \otimes \id_V) \circ s_V) \otimes \id_{V^*} ) \circ \lcoev_V \right) \otimes \id_W \right) \\*
  &\hspace{\parindent}= f_\calC(\eta) \circ \left( \id_U \otimes \left( ((((\intL \circ \epsilon_{\one}) \otimes \id_V) \circ s_V) \otimes \id_{V^*} ) \circ \lcoev_V \right) \otimes \id_W \right) \\*
  &\hspace{\parindent}= f_\calC(\eta) \circ \left( \id_U \otimes \left( ((((i_G \circ f_\intL) \otimes \id_V) \circ s_V) \otimes \id_{V^*} ) \circ \lcoev_V \right) \otimes \id_W \right) \\*
  &\hspace{\parindent}= \eta_G \circ \left( \id_U \otimes \left( (((f_\intL \otimes \id_V) \circ s_V) \otimes \id_{V^*} ) \circ \lcoev_V \right) \otimes \id_W \right).
 \end{align*}
 Therefore our claim is a direct application of the equality we just established for $U$ and $W$ determined by the source, $Z$ determined by the target, and $\eta$ determined by the cycle $C$ of $T$.
\end{proof}

\addtocontents{toc}{\protect\setcounter{tocdepth}{2}}

\end{document}